\documentclass[pdflatex,sn-mathphys-ay]{sn-jnl}

\usepackage{graphicx}%
\usepackage{multirow}%
\usepackage{amsmath,amssymb,amsfonts}%
\usepackage{amsthm}%
\usepackage{mathrsfs}%
\usepackage[title]{appendix}%
\usepackage{xcolor}%
\usepackage{textcomp}%
\usepackage{manyfoot}%
\usepackage{booktabs}%
\usepackage{algorithm}%
\usepackage{algorithmicx}%
\usepackage{algpseudocode}%
\usepackage{listings}%
\usepackage{subcaption,paralist}

\theoremstyle{thmstyleone}%
\newtheorem{theorem}{Theorem}
\newtheorem{proposition}[theorem]{Proposition}%
\newtheorem{lemma}[theorem]{Lemma}
\newtheorem{conjecture}{Conjecture}

\theoremstyle{thmstyletwo}%
\newtheorem{remark}{Remark}%

\theoremstyle{thmstylethree}%

\def\N{\mathbb{N}}

\def\R{\mathbb{R}}

\begin{document}

\title[Controllability of Lotka -- Volterra strongly competitive systems]{Controllability of diffusive Lotka -- Volterra strongly competitive systems under boundary constrained controls}


\author*[1]{\fnm{Elisa} \sur{Affili}}\email{elisa.affili@univ-rouen.fr}

\author[2,3,4]{\fnm{Enrique} \sur{Zuazua}}\email{enrique.zuazua@fau.de}
\equalcont{Orcid ID: (Elisa Affili) https://orcid.org/0000-0002-0634-7158, 
(Enrique Zuazua) https://orcid.org/0000-0002-1377-0958}

\affil*[1]{\orgdiv{Univ Rouen Normandie, CNRS, Normandie Univ, LMRS UMR 6085},  \orgaddress{\postcode{F-76000}, \city{Rouen}, \country{France}}}

\affil[2]{\orgdiv{Chair for Dynamics, Control, Machine Learning, and Numerics (Alexander von Humboldt Professorship), Department of Mathematics, Friedrich–Alexander-Universit\"at Erlangen–N\"urnberg, 91058 Erlangen, Germany}}

\affil[3]{\orgdiv{Departamento de Matem\'aticas, Universidad Aut\'onoma de Madrid, 28049 Madrid, Spain}}

\affil[4]{\orgdiv{Chair of Computational Mathematics, Universidad de Deusto. Av. de las Universidades, 24, 48007 Bilbao, Basque Country, Spain}}


\abstract{We investigate the controllability of the competition-diffusion Lot\-ka\---Vol\-ter\-ra system. Our primary focus is on the one-dimensional setting with Dirichlet boundary controls, interpreted as ecological management policies regulating the density of species at the habitat boundaries and satisfying bilateral constraints.

We show that the system can be steered from any initial state to a constant steady state representing the extinction of the less competitive species. In contrast, we prove that controllability toward a steady state where the more competitive species vanishes is generally not achievable when the inter-species competition rates are too unbalanced. This obstruction is due to the existence of barrier solutions, which we explicitly construct based on the spectral properties of the associated reaction -- diffusion operators.

Our theoretical results are illustrated through numerical simulations and are accompanied by a discussion of open problems and potential directions for future research.
}




\maketitle

\section{Problem formulation}

In this paper, we explore the possibility of steering the diffusive dynamics of two competing populations within a confined domain -- a prototypical ecological model -- through boundary controls that emulate practical interventions, such as animal reintroduction, hunting, or migration policies. These controls aim to manage species movement and regulate population densities across territorial boundaries.

A central challenge lies in ensuring that population densities comply with bilateral constraints, which are essential for maintaining biologically meaningful conditions. Designing control strategies under these constraints is particularly demanding, as they take the form of state and control limitations. In practice, this means that both the population states and the applied control inputs must remain within prescribed bounds that reflect realistic ecological conditions. Addressing this issue calls for the development of tailored control methodologies capable of operating effectively within such constraints.

More precisely, we study a competition-diffusion Lotka -- Volterra system involving two interacting populations. 
Each equation in the system captures the spatial dynamics, reproduction, and mortality of a population, with inter-species coupling arising from competition for limited natural resources.

For clarity of exposition and simplicity in the calculations, we focus on the one-dimensional setting. However, the core ideas and techniques extend naturally to higher dimensions; particularly, the extension to radially symmetric domains is straightforward. 
Our analysis is based on Dirichlet boundary controls, which are classical in control theory. 
Nevertheless, as is customary, the controllability results can be reinterpreted under different types of boundary conditions -- a point we revisit later in this section.

Let $y_1(x,t)$ and $y_2(x,t)$ denote the population densities of the two species, defined on the spatial interval $(0, L)$, with $L > 0$, and, considered over an infinite time horizon $t > 0$. The system under study is given by
\begin{equation}\label{sys:lv}
\left\{
\begin{array}{ll}
\partial_t y_1 -  \partial_{xx}^2 y_1 = y_1(1 - y_1 - a y_2), & \text{for} \ x\in(0,L), \ t>0,\\
\partial_t y_2 - \partial_{xx}^2 y_2 = y_2(1 - b y_1 - y_2), & \text{for} \ x\in(0,L), \ t>0,\\
y_1(x, 0)=y_1^0, \ y_2(x, 0)=y_2^0, & \text{for} \ x\in (0,L),
\end{array}\right.
\end{equation}
where 
\begin{equation}\label{hyp:sc}
a>0 \quad \text{and} \quad b>\max \{a, 1\},	
\end{equation}
are the competition coefficients,  and $y_1^0,y_2^0\in L^\infty((0,L);[0,1])$ are bounded initial data.

The system dynamics is controlled through \emph{boundary controls}
\begin{equation}\label{hyp:spaces}
u_1(x,t), \ u_2(x,t) \in L^{\infty}(\{0,L\}\times \R^+; [0,1])     
\end{equation}
representing the densities of each population at the boundary. Thus, we impose
\begin{equation}\label{sys:lv2}
\left\{
\begin{array}{ll}
y_1(x,t)=u_1(x,t),  & \text{for} \ x=0, L, \ t>0, \\
y_2(x,t)=u_2(x,t), & \text{for} \ x=0,L, \  t>0.
\end{array}\right.
\end{equation}

By assuming \eqref{hyp:sc}, we made two choices:
\begin{itemize}
\item Biologically,  $a < b$ means that  the first population is competitively stronger than the second one. Note that the opposite case $a > b$ is symmetric and leads to analogous results.  The degenerate case $a = b$ is excluded from our analysis and briefly discussed in the perspectives.
	\item	We further assume $b > 1$ to exclude the regime of weak competition, in which the two species can coexist. In this case, the system exhibits qualitatively different behavior; see \cite{sonego2024control} for a detailed discussion.\end{itemize}

Note that, for simplicity, all coefficients in system \eqref{sys:lv} -- except for the competition terms -- are normalized to 1. Remarks on the influence of diffusion coefficients and reproduction rates, which appear as multiplicative factors in the nonlinear terms, are provided in Remark \ref{rmk:d} in the following section.

\medskip

Notice that, for both equations, the state upper bound $y=1$  represents the so-called \textit{carrying capacity}, that is, the maximal density of individuals that the renewable resources of the environment can sustain.

Accordingly, solutions are assumed to satisfy the bilateral constraints $0 \le y \le 1$.  More precisely, \begin{equation*}
0 \leq y_1(x,t), y_2(x,t) \leq 1, \qquad \text{for all} \ x\in(0, L), t>0.
\end{equation*}
For this to hold, initial data and controls  need to fulfil the same constraints as well:
\begin{equation}\label{hyp:bounds}
\begin{split}
	0 \leq y_1^0(x), y_2^0(x) \leq 1, \qquad &\text{for all} \ x\in(0, L),  \\
	0 \leq u_1(x,t), u_2(x,t) \leq 1, \qquad & \ x=0, L, \ t>0.
\end{split}
\end{equation}
The set $[0,1]\times[0,1]$ is invariant under the flow of the competition-diffusion Lotka -- Volterra system \eqref{sys:lv}, under the bilateral constraints on the controls above.
This is due to the monotonicity structure of the system, as we recall in Subsection \ref{ss:preliminars}.

We say that \eqref{sys:lv} is \emph{asymptotically controllable} towards a steady state $(\bar{y}_1(x), \bar y_2(x))$ if for any initial condition $(y_1^0,y_2^0)$ (satisfying \eqref{hyp:bounds}) there exist controls $u_1,u_2\in L^{\infty}(\{0,L\}\times \R^+; [0,1])$
such that
$$(y_1(t,x),y_2(t,x)) \longrightarrow (\bar{y}_1(x), \bar y_2(x))$$
uniformly in $[0, L]$ as $t \to +\infty$.

Our main goal is to investigate the boundary controllability of   the diffusive system \eqref{sys:lv}, under constraints, towards steady state solutions (thus, not depending on $x$ nor $t$). 
Under the parameter assumptions in \eqref{hyp:sc}, in the absence of diffusion, the system becomes  
\begin{equation}\label{sys:odes}
\left\{
\begin{array}{ll}
\dot{w}_1 = w_1(1 - w_1 - a w_2), & \text{for}  \ t>0,\\
\dot{w}_2= w_2(1 - b w_1 - w_2), & \text{for} \ t>0, \\
(w_1,w_2)(0)=(w_1^0,w_2^0),
\end{array}\right.
\end{equation}
where $\dot{w}_1$ and $\dot{w}_2$ denote the time-derivative of $w_1$ and $w_2$ respectively.
It admits the following equilibria (i.e., time-independent solutions):
\begin{itemize}
	\item	$(0,1)$, which is locally attractive, 
	\item	$(1,0)$, which is locally attractive if $a > 1$,
	\item	$(0,0)$, which is  a saddle point if $0 < a \leq 1$ and is repulsive if $a >1$,
	\item	a coexistence equilibrium $(w_1^*, w_2^*)$ appears only when $a > 1$, and it is given by
$$(w_1^*, w_2^*) := \left( \frac{a - 1}{ab - 1}, \frac{b - 1}{ab - 1} \right).$$
This point is a saddle with a one-dimensional stable manifold. 
Notice that for $a<1$, this equilibrium falls outside $[0,1]\times[0,1]$, while for $a=1$ it coincides with $(0,1)$.
\end{itemize}
The equilibria of the dynamical system \eqref{sys:odes} are the only values that the steady states of system \eqref{sys:lv} can take. Thus, we investigate controllability towards the steady states $(1,0)$, $(0,1)$, $(0,0)$, and, whenever $a>1$, $(w_1^*,w_2^*)$.

Studying extinction phenomena -- cases where one or both components approach zero -- is particularly relevant, as it corresponds to the possibility of eradicating an undesired species from the controlled area, such as a harmful insect or an invasive alien species. Invasive alien species represent a major driver of biodiversity loss, as has been emphasized by experts for many years (see \cite{pyvsek2020scientists}). Therefore, their eradication should be considered a priority for the preservation of local ecosystems.

Specifically, we aim to clarify the role played by the competition coefficients and the length of the spatial interval $L$ in determining the possibility of achieving these control properties. We will construct some barrier solutions preventing the controllability to the steady states $(0,1)$ and $(w_1^*, w_2^*)$.
These constructions rely on a deep understanding of the properties of the steady states, as well as on the use of the comparison principle for the Lotka -- Volterra system.

\begin{remark}{ (Neumann control).}
To keep the presentation concise and accessible, we restrict our analysis to Dirichlet boundary controls. However, the results we obtain can also be interpreted under other boundary conditions, such as Neumann controls, where the boundary input prescribes the value of the normal derivative.

In fact, once controllability has been established using Dirichlet controls, the corresponding Neumann boundary controls can be recovered by simply taking the Neumann traces of the solution at $x = 0$ and $x = L$.

 Moreover, when the initial data are even with respect to $L'=L/2$, because of the symmetry of the system, the formulated problem is equivalent to controlling \eqref{sys:lv} on the interval $(0, L')$  and the mixed boundary controls 
\begin{equation*}
\left\{
\begin{array}{ll}
y_1(0,t)=u_1(t),  & \text{for} \ t>0, \\
y_2(0,t)=u_2(t), & \text{for} \  t>0, \\
\partial_x y_1(L',t)= 0,  & \text{for} \ t>0, \\
\partial_x y_2(L',t)= 0,  & \text{for} \ t>0.
\end{array}\right.
\end{equation*}
\end{remark}

\section{Main results}\label{s:main}

\subsection{Asymptotic controllability}
First, we present a series of positive results on the asymptotic controllability as $t \to \infty$ to the steady states $(1,0)$, $(0,1)$ and negative and positive results for the steady state $(0,0)$.
By $B(0,1)$ we denote the basin of attraction of the equilibrium $(0,1)$ for the dynamical system \eqref{sys:odes}; notice that this set is void for $a<1$.

\begin{theorem}\label{thm:strong}
{\it 	Let $y_1^0,y_2^0: (0,L) \to [0, 1] $ be a pair of measurable initial conditions and let $(y_1,y_2)$ be the solution of system \eqref{sys:lv} with initial data $(y_1^0,y_2^0)$, and   boundary conditions as in \eqref{sys:lv2} to be chosen, and with parameters as in \eqref{hyp:sc}. Then:
	\begin{enumerate}
		\item     $(y_1,y_2) \to (1,0)$ as $t\to\infty$ with boundary controls $(u_1,u_2)\equiv (1,0)$;
		\item If $L\leq \pi$,  $(y_1,y_2) \to (0,1)$ as $t\to\infty$ with the boundary controls $(u_1,u_2)\equiv(0,1)$;
		\item  If $a>1$, $(y_1,y_2) \to (0,1)$ as $t\to\infty$, if $(y_1^0(x),y_2^0(x)) \in B(0,1)$ for all $x\in(0,L)$, with the boundary controls $(u_1,u_2)\equiv (0,1)$;
		\item If $L\leq \pi$,   $(y_1,y_2) \to (0,0)$ as $t\to\infty$, with the boundary controls $(u_1,u_2)=(0,0)$;
		\item If $L>\pi$, for all $(y_1^0, y_2^0)\ne (0,0)$, there exists no boundary controls $(u_1,u_2)\in L^{\infty}(\{0,L \}  \times \R^+; [0,1])$ such that  $(y_1,y_2) \to (0,0)$ as $t\to\infty$.
	\end{enumerate}}
\end{theorem}

We now provide a series of remarks to clarify and interpret the statements of Theorem \ref{thm:strong}.
\begin{remark}
	In the statements 1-4, we choose the simplest and more natural controls entailing the desired behavior, namely, the time-independent ones. However, some controls that asymptotically take these values lead to the same result. The asymptotic controllability is then a consequence of the asymptotic attractiveness or stability of the system, with the chosen time-independent boundary conditions. This stability properties will be proved relying mainly on the comparison principle. In the proofs, we will sometimes construct some more general controls (see also Remark \ref{rmk:d}). 
\end{remark}

\begin{remark}
	In all the statements, we presented only asymptotic controllability results. The targets $(1,0)$, $(0,1)$ and $(0,0)$, which saturate the constraints, even when they can be reached asymptotically, cannot be reached in finite time, as mentioned above and further explained as follows.
		
The competition-diffusion Lotka -- Volterra systems enjoys the Comparison Principle in the following sense: if we have two solutions $(\underline{y}_1,\underline{y}_2)$ and $(\overline{y}_1,\overline{y}_2)$ such that $\underline{y}_1 \le \overline{y}_1$ and $\underline{y}_2 \ge \overline{y}_2$ on the parabolic boundary $\left( (0,L)\times \{0\} \right) \cup \left( \{0, L\}\times \R^+ \right)$, then it holds that  $\underline{y}_1 \le \overline{y}_1$ and $\underline{y}_2 \ge \overline{y}_2$ in all the cylinder $[0, L]\times \R^+$. Moreover, if equality holds in one point in the interior of  $(0, L)\times \R^+$, then equality must hold for all the points in $(0, L)\times \R^+$ (see also the preliminars in Section \ref{ss:preliminars} and in particular Theorem \ref{thm:cp}). 
	
	Now, it is easy to notice that $(0,1)$ and $(1,0)$ are solutions for the competition-diffusion Lotka -- Volterra system \eqref{sys:lv} with the corresponding boundary conditions, actually coinciding componentwise with the values of the solution itself. Moreover, any other solution $(y_1, y_2)$, regardless what the admissible applied boundary controls are, satisfies $0 \leq y_1, y_2 \leq 1$, as discussed in the introduction. Hence, by the Comparison Principle, the solution $(y_1, y_2)$ cannot coincide with neither $(0,1)$ nor $(1,0)$ in finite time  without coinciding for all times. This means that controllability in finite time cannot be achieved. This is consistent with the need of trajectories to oscillate around the target before reaching it in finite time. 
    
    For the case of $(0,0)$, we can directly apply the Comparison Principle for parabolic equations to each differential equation of \ref{sys:lv}, obtaining that only asymptotic controllability is possible (and not even in all the cases). This was also the case for the monostable and bistable scalar equation in \cite{pouchol2019phase} and \cite{domenec1}.
\end{remark}

\begin{remark}\label{rmk:d}
	The proof of the first statement in Theorem \ref{thm:strong} employs a planar 
 travelling front, that is 
 $(\alpha, \beta)(x,t)=(Y_1, Y_2)(x-ct)$; 
connecting $(0,1)$ (for $t\to -\infty$) to $(1,0)$  (for $t\to +\infty$). In fact, a suitable translation of the travelling wave is ``below'' the solution $(y_1, y_2)$ with boundary controls $(u_1,u_2)=(1,0)$, meaning that $y_1 \geq \alpha$ and $y_2 \leq \beta$ in the parabolic boundary $\left( (0,L)\times \{0\} \right) \cup \left( \{0, L\}\times \R^+ \right)$; hence, by the Comparison Principle, the ordering is preserved in the cylinder $(0,L)\times \R^+$. Since  $(\alpha, \beta)$ converges to $(1,0)$  for $t\to +\infty$, so does the solution $(y_1, y_2)$, meaning that the travelling wave ``pushes'' the solution $(y_1, y_2)$ to $(0,1)$. This travelling front argument is needed to prove global asymptotic stability of $(1,0)$ with suitable boundary data. Once this is proven, it is trivial to see that $(u_1,u_2)=(1,0)$ is a possible choice.

	The existence of a travelling front  for \eqref{sys:lv} under the hypothesis \eqref{hyp:sc}, moving with positive speed from $(0,1)$ to $(1,0)$  was proved in  \cite{guo2013sign}. More precisely, in \cite{guo2013sign} the competition-diffusion Lotka -- Volterra system was considered
	\begin{equation*}
		\left\{
		\begin{array}{ll}
			\partial_t y_1 -  \partial_{xx}^2 y_1 = y_1(1 - y_1 - a y_2), & \text{for} \ x\in\R, \ t>0,\\
			\partial_t y_2 - d\partial_{xx}^2 y_2 = ry_2(1 - b y_1 - y_2), & \text{for} \ x\in\R, \ t>0,
		\end{array}\right.
	\end{equation*}
	where, in addition to the competition coefficients $a$ and $b$ satisfying \eqref{hyp:sc},    the diffusion coefficient $d>0$  and the reproduction rate $r>0$ were chosen for the second population. The authors were able to completely characterize the sign of the speed of the traveling wave in the cases $r=d$ (which is our case, too) and $r=d/4$. Other conditions entailing a positive sign of the speed were given in \cite{guo2019sign, ma2019speed, ma2020bistable}. The first statement of Theorem \ref{thm:strong} continues to hold within these settings.\end{remark}

\begin{remark}\label{rmk:st245}
	In the statements 2, 4 and 5 of Theorem \eqref{thm:strong}, we see that $L=\pi$ is a threshold. 
	In fact, when $L\leq\pi$, there is no nonnegative steady state solution for the logistic equation with null Dirichlet boundary conditions
	\begin{equation}\label{eq:logistic}
		\partial_t \theta-\partial_{xx}^2 \theta=\theta(1-\theta), \quad \theta(0,t)=\theta(L,t)=0,
	\end{equation}
apart from the trivial solution $\theta \equiv 0.$ Therefore all solutions generated by initial data taking values in $[0, 1]$ and null Dirichlet boundary conditions converge to $0$. We then notice that the solutions of the logistic equation \eqref{eq:logistic} are supersolutions for a single equation  of the competition-diffusion Lotka--Volterra system. Then, by the classical Comparison Principle for parabolic equations (see e.g. \cite{protter2012maximum}), since $\theta$ goes to 0, then also $y_1 \leq \theta$ goes to 0. This proves the second statement of Theorem \ref{thm:strong}. For the fourth statement, we compare separately $y_1 \leq \theta$ and $y_2\leq \theta$, and since $\theta$ converges to 0, then both $y_1$ and $y_2$ converge to 0.
	
	However, when $L>\pi$, the logistic equation \eqref{eq:logistic} has a nontrivial steady state solution $\Theta$, and  all solutions generated by positive initial data of the parabolic equation converge to $\Theta$. Now, to prove the fifth statement, we compare the equation satisfied by $y_1+y_2$ with Dirichlet conditions with the logistic equation, and we find that  $y_1+y_2$ also converges to some function larger than $\Theta$. This means that $y_1$ and $y_2$ cannot converge to 0 simultaneously.
	Hence, $(0,0)$ is no longer reachable.
\end{remark}

\begin{remark}
	We now comment on statement 3. For $a< b$, in an unbounded environment, the first species would be dominant and eventually wipe out the second one. 
	Here, we show that for some particular initial data (lying in the green set of Figure \ref{fig:odes}) -- describing the basin of attraction of the equilibrium $(0, 1)$ for the non-diffusive Lotka -- Volterra model, the system can be controlled to the ``unnatural'' equilibrium $(0,1)$. 
	
	This set of initial data is likely not optimal. This is a topic that requires further investigation.

In the following section, we show that not all initial data can be steered to the steady state $(0,1)$. Specifically, for certain combinations of the parameters $a, b$, and the domain length $L$, the boundary controls are insufficient to counterbalance the competitive advantage of the first population. As a result, extinction of the first species cannot be achieved, and the steady state $(0,1)$ is not always attainable.
\end{remark}

\subsection{Existence of barriers and non-controllability results}
The impossibility of reaching some of the state states of the system can be explained through the existence of \emph{barrier solutions} that controlled trajectories may not cross, regardless of the values of the controls within the given constraints. 

We call a \emph{barrier} a solution $(\phi, \psi)$ to the system 
\begin{equation}\label{sys:barrier}
\left\{
\begin{array}{ll}
-  \phi'' = \phi(1 - \phi - a \psi), & \text{for} \ x\in(0,L), \\
- \psi'' =  \psi(1 - b \phi - \psi), & \text{for} \ x\in(0,L), \\
\phi(x)=0, \ \psi(x)=1,  & \text{for} \ x=0, L, \\
0 <\phi(x), \ \psi(x) < 1, & \text{for} \ x\in (0,L).
\end{array}\right.
\end{equation}
Notice that the trivial steady state $(0,1)$ is excluded, meaning that the first population has a positive density on the interval.
So, the steady state with $(0,1)$ boundary condition is not unique, because of the non-linearity. 
Given that in this setting we only control boundary values,  this implies that we are not always able to reach the desired target. 
From the biological interpretation, it means that
the superiority of the first population cannot be compensated by the action of the boundary control.
In fact, even under the most unfavourable boundary conditions -- namely, zero density for the first population and full density for the second -- the first population persists.

Moreover, travelling waves give us important insight on the controllability properties. From (\cite{gardner1982existence, kan1996existence}) we know that, for $a>1$, there exists a travelling front connecting $(w_1^*, w_2^*)$ to $(0,1)$. So, if one is able to control (exactly) to $(w_1^*, w_2^*)$, then one is also able to control asymptotically to $(0,1)$. 
This suggests that if $(0,1)$ is not reachable, then $(w_1^*, w_2^*)$ is not reachable either.
However, since this construction requires a few technical details, we will employ another quicker argument (in Lemma \ref{lemma:barrier}). 
In Figure \ref{fig:waves}, we represent the known travelling fronts connecting equilibria.

Note that, as mentioned above (\cite{sonego2024control}),  when the equilibrium $(w_1^*, w_2^*)$  lies strictly within the interior of the admissible set $(0,1) \times (0,1)$, one might conjecture that if it is asymptotically reachable, then it should also be attainable in finite time via suitable controls that respect the imposed constraints. However, the arguments presented above demonstrate that even asymptotic reachability of $(w_1^*, w_2^*)$  cannot be expected.

\begin{figure}
	\centering
	\includegraphics[width=0.4\linewidth]{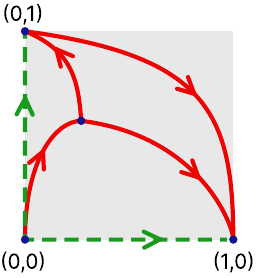}
	\caption{
    This is a heuristic representation meant to visualize the existence of travelling fronts for $a>1$ for the diffusive Lotka -- Volterra system.
    Blue dots represent the constant equilibria, with red arrows denoting the travelling fronts connecting the equilibria, and with green arrows the travelling fronts for the case in which one of the two population is not present (and the remaining population density satisfies a logistic equation). The direction of the arrows points towards the state invading  space with positive speed.}
	\label{fig:waves}
\end{figure}

To state our next results more compactly, we introduce the following set of initial data, that depends on the barrier:
\begin{multline}
    \mathcal{I}_{\phi,\psi} := \{ (y_1^0, y_2^0)\in L^\infty((0,L); [0,1]) \ | \  y_1^0 (x) \geq\phi (x), \\ y_2^0(x) \leq \psi(x) \ \text{ for all} \ x\in [0, {L}]  \}.
\end{multline}

Then we have:

\begin{theorem}\label{thm:implies}
	\emph{Suppose there exists a barrier $(\phi, \psi)$  for some parameters $\bar{a}$, $\bar{b}$, $\bar{L}$. Then, for such parameters, the system \eqref{sys:lv} with boundary controls satisfying \eqref{hyp:spaces}-\eqref{hyp:bounds} is not controllable to $(0,1)$ for any initial data $(y_1^0, y_2^0)\in \mathcal{I}_{\phi,\psi}$. 
	Moreover, if $\bar{a}>1$, the same system is not controllable to the coexistence equilibrium $(w_1^*,w_2^*)$ for any initial data $(y_1^0, y_2^0)\in \mathcal{I}_{\phi,\psi}$.
	}
\end{theorem}

The existence of a barrier solution depends on the parameters of the system. The following two dichotomous theorems analyse the dependence on one parameter, once the other two are fixed. 

The larger the competition coefficient $b$ is, the more difficult it is for the second population to survive. This suggests the following result:

\begin{theorem}\label{thm:b}
		\emph{For any given $a>0$ and $L>\pi$, there exits $b^*>1$ such that:
		\begin{itemize}
			\item  for all $1< b < b^*$, for any initial data $y_1^0, y_2^0 \in L^\infty((0,L);[0,1])$, the system is asymptotically controllable to $(0,1)$ with boundary controls $(u_1,u_2)\equiv(0,1)$.
			\item for all $b> b^*$, there exists a barrier, therefore, for initial data $(y_1^0, y_2^0) \in \mathcal{I}_{\phi,\psi}$, the system is not controllable to $(0,1)$. If moreover $a>1$, for $b\geq b^*$, for initial data $(y_1^0, y_2^0) \in \mathcal{I}_{\phi,\psi}$, the system is not controllable neither to $(w_1^*, w_2^*)$ nor to $(0,1)$.
		\end{itemize} }
\end{theorem}

Also, when the competition coefficient $a$ is too small, the growth of the first population cannot be controlled by the second population. Hence, we get the following:

\begin{theorem}\label{thm:a}
	\emph{For any given $b>1$ and $L>\pi$, there exits $a^*\geq 1-\frac{\pi^2}{L^2}$ such that:
	\begin{itemize}
		\item  for all $a> a^*$, for any initial data $y_1^0, y_2^0 \in L^\infty((0,L);[0,1])$, the system is asymptotically controllable to $(0,1)$ with controls $(u_1,u_2)\equiv (0,1)$. 
		\item for all $0< a < a^*$, there exists a barrier, therefore, for initial data $(y_1^0, y_2^0) \in \mathcal{I}_{\phi,\psi}$, the system is not controllable to $(0,1)$. If moreover $a^*>1$, for all $1<a\leq a^*$, for initial data $(y_1^0, y_2^0) \in \mathcal{I}_{\phi,\psi}$, the system is not controllable to $(0,1)$ and to $(w_1^*, w_2^*)$.
	\end{itemize} }
\end{theorem}

\begin{remark}
	We also observe that non-controllability to $(0,1)$ can only occur when $L >\pi$. In fact, when $L\leq \pi$, the system can be controlled to $(0,1)$ from every admissible initial datum, as stated in point 2 of Theorem \ref{thm:strong} and explained in Remark \ref{rmk:st245}. In this case, barriers cannot exist.
\end{remark}

\begin{remark}
    We cannot always guarantee that $a^*>1$. However, due to Theorem \ref{thm:b}, we know that it is possible to have non-controllability for any $a>1$, provided that $b$ is large enough. 
\end{remark}

\section{State of the art}
 When considering the competition-diffusion Lotka--Volterra system on the real line and in the absence of controls, several distinct asymptotic behaviors may arise: \textit{extinction}, when one species vanishes asymptotically uniformly (i.e., it converges to $0$); \textit{spreading}, when a species converges to the carrying capacity of the environment ($\equiv 1$ in our setting); and \textit{coexistence}, when neither species undergoes extinction.

It is known (see \cite{hirsch1983differential, hirsch1988systems}) that all solutions to \eqref{sys:lv} with time-independent boundary conditions converge, as time tends to infinity, to a steady state--i.e., a time-independent solution of \eqref{sys:lv}. Homogeneous steady states (i.e., independent also of the spatial variable $x$) are particular examples of such solutions.

However, the asymptotic behavior of the competition-diffusion Lotka--Volterra system \eqref{sys:lv} in bounded domains remains poorly understood, even under classical boundary conditions such as Neumann ones, because of the possible multiple steady states sharing the same boundary conditions. For results concerning convergence to constant steady states, when initial data lies in the basin of attraction (in the sense of the associated ODE system \eqref{sys:odes}), we refer to \cite{ninomiya1997separatrices, shim2002domains}. Additionally, \cite{iida1998diffusion} illustrates the phenomenon of diffusion-induced extinction: starting from spatially homogeneous initial data lying in the basin of attraction of one equilibrium of the ODE system, the solution of the diffusive system may converge to a different steady state. Several classical studies have investigated the conditions on coefficients and boundary conditions that allow for stable coexistence steady states; see \cite{blat1984bifurcation, cosner1984stable, matano1983pattern, shigesada1984effects}. The asymptotic behavior in the strong competition regime (i.e., as competition coefficients tend to infinity) has also been studied in \cite{conti2005asymptotic, terracini2019spiraling}.

Much of the literature on the competition-diffusion Lotka--Volterra system focuses on the whole space. In this setting, when $a\neq b >1$, it is known that ultimately only one species survives. A particularly important class of entire solutions in this context is given by \emph{travelling waves} (or \emph{travelling fronts})--solutions that propagate at a constant speed while maintaining their shape and connecting two distinct constant equilibria as $x \to \pm \infty$. The existence of travelling waves was established in \cite{gardner1982existence}, \cite{kan1996existence}, and \cite{tang1980propagating}. Since these foundational works, many efforts have been devoted to studying the properties of such solutions. In particular, we draw on results from \cite{guo2013sign}, which analyzes the sign of the propagation speed of certain travelling waves.

Knowledge of travelling waves also provides insight into the spreading speed of the dominant species in unbounded domains, especially when starting from compactly supported initial data; see \cite{kan1997fisher}. A prevailing conjecture suggests that the spreading front of solutions converges to a travelling wave profile, analogous to what occurs in scalar KPP-type equations.

For semilinear parabolic systems modeling species interaction, null controllability in finite time has been achieved (see \cite{balch}). However, these results typically neglect natural constraints on both the state and the controls.

Addressing these constraints requires complementary methodologies. For instance, \cite{dario}, \cite{domenec1}, and \cite{pouchol2019phase} consider the heat equation and reaction--diffusion equations with monostable and bistable nonlinearities. Their approach leverages the structure of travelling wave solutions and long-time controllability along paths of steady states.

The controllability of the competition-diffusion Lotka--Volterra system for two species in the weak interaction regime was investigated in \cite{sonego2024control}. In contrast, the present work focuses on the regime of strong competition.

\section{Numerical simulations}

To support Theorems \ref{thm:b}, we carried out a series of numerical simulations, the results of which are presented in Figures \ref{fig:base} and \ref{fig:b}. The simulations were performed in MATLAB using the optimization package CasADi. The equations were discretized using an explicit finite difference scheme. Although the optimization is conducted over a finite time interval, the chosen time horizons are sufficiently long to clearly exhibit the system's asymptotic behavior.

Figure \ref{fig:base} demonstrates that the system can be successfully controlled to the steady state $(0,1)$ from the initial condition $(1,0)$, using boundary controls set to $(0,1)$. As the competition coefficient $b$ increases, the simulations reveal the emergence of barriers. In Figure \ref{fig:base}, with $b = 2.6$, controllability is achieved. 

However, for the same values of $a$ and $L$, Figure \ref{fig:b} (with $b = 3.5$) shows that the system stabilizes at a non-trivial steady state under the same boundary control $(0,1)$, indicating the presence of a barrier. Similar phenomena are observed in simulations targeting the coexistence steady state $(w_1^*,w_2^*)$.

\begin{figure}[htbp]
    \centering
    \begin{subfigure}{0.49\textwidth}
        \centering
        \includegraphics[width=\linewidth]{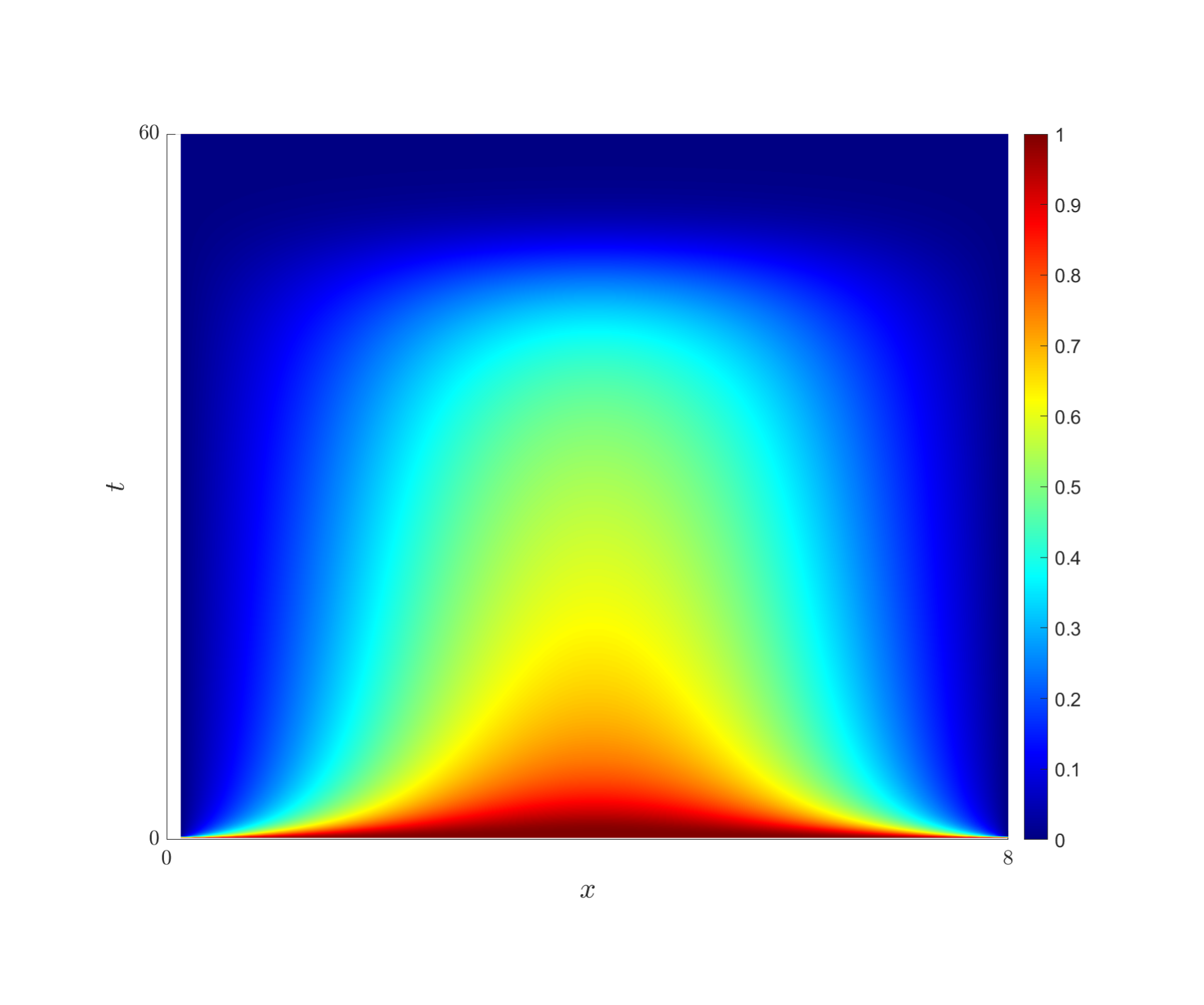}
        \caption{Solution $y_1$}
    \end{subfigure}
    \hfill
    \begin{subfigure}{0.49\textwidth}
        \centering
        \includegraphics[width=\linewidth]{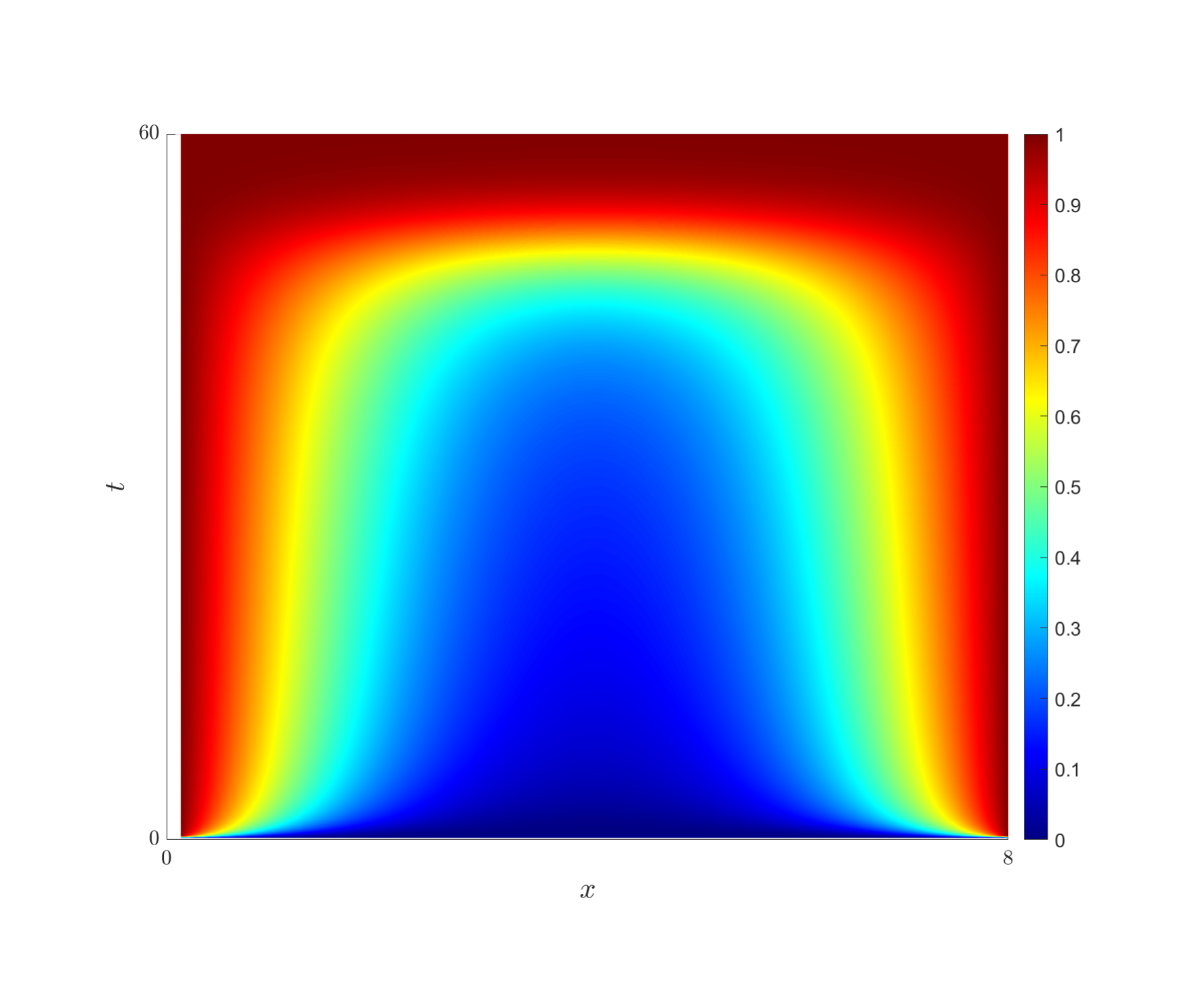}
        \caption{Solution $y_2$}
    \end{subfigure}
    \caption{Controllability to the steady state $(0,1)$ from the initial state $(1,0)$ with $L=8$, $a=1.5$, $b=2.6$, $T=60$, computed in MATLAB using the \texttt{CasADi} package. The simulation shows no barrier effects and a good approximation of the target state within the given time frame, which appears to be more than sufficient.}
	\label{fig:base}
\end{figure}

\begin{figure}[htbp]
    \centering
    \begin{subfigure}{0.48\textwidth}
        \centering
        \includegraphics[width=\linewidth]{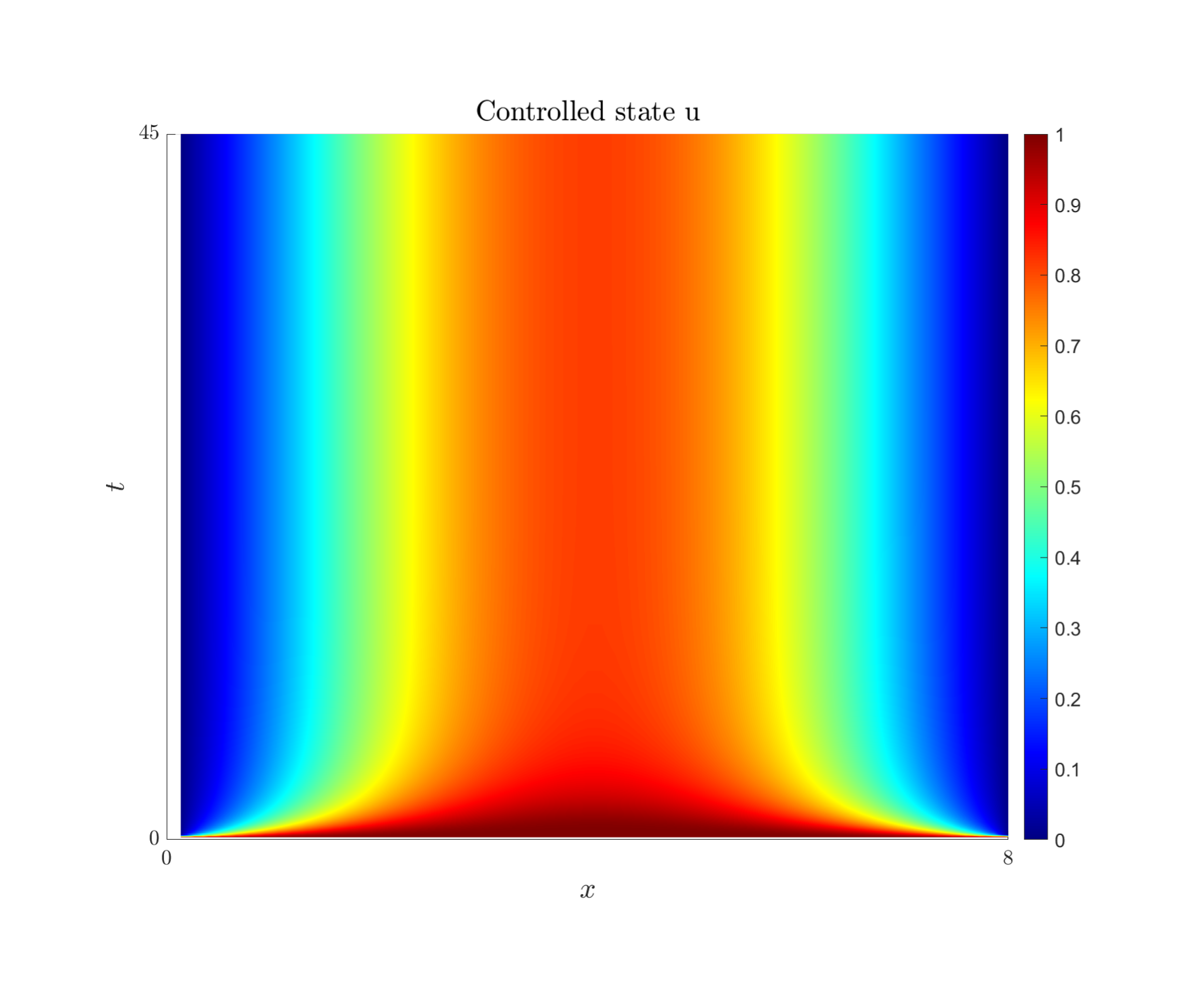}
        \caption{Solution $y_1$}
    \end{subfigure}
    \hfill
    \begin{subfigure}{0.48\textwidth}
        \centering
        \includegraphics[width=\linewidth]{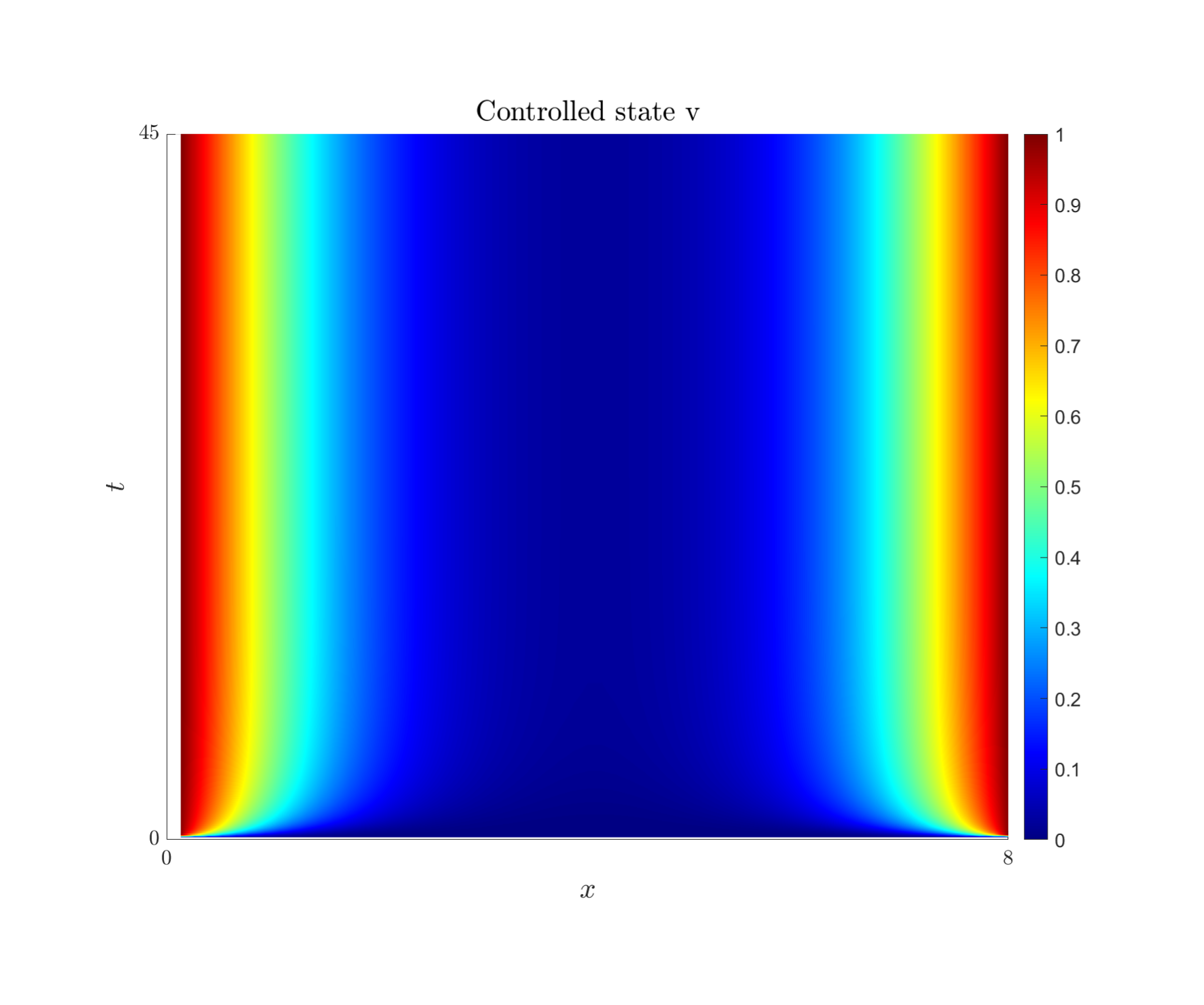}
        \caption{Solution $y_2$}
    \end{subfigure}
    \caption{Controllability to the steady state $(0,1)$ from the initial state $(1,0)$ with $L = 8$, $a = 1.5$, $b = 3.5$, and $T = 45$, simulated in \textsc{Matlab} using the \texttt{CasADi} package. Here, we observe that the target is not well approximated, and the system appears to approach a steady state associated with the boundary data $(0,1)$, thus a barrier.}
    \label{fig:b}
\end{figure}

\section{Conjectures}

Intuitively, the larger the spatial domain $L$, the weaker the influence of boundary controls on the overall behavior of the solution. As it occurs with the parameters $a$ and $b$, one may expect the emergence of barriers for sufficiently large values of $L$. In fact, the existence of a barrier exhibits a monotonic dependence on $L$. Additional numerical simulations confirm this intuition: for large enough intervals, a barrier indeed appears. This is clearly illustrated by comparing Figure \ref{fig:base} (with $L = 8$) and Figure \ref{fig:L} (with $L = 16$). In the former, the system can be controlled to the target steady state $(0,1)$ from the initial condition $(1,0)$, while in the latter, the solution stabilizes at a nontrivial steady state corresponding to the boundary data $(0,1)$, indicating the presence of a barrier. 

Nevertheless, providing a rigorous proof of this behaviour remains a significant challenge. In fact, unlike what we have with respect to the coefficients $a$ and $b$, we are lacking the analytical construction of a barrier for some suitable $\tilde{L}$. This is due to the fact that very little is known about the precise structure of solutions or the shape and properties of travelling waves, which would be the starting point to the construction of a barrier in a very large interval.

We expect the following result to hold; however, establishing it rigorously requires substantial further analysis:
\begin{conjecture}\label{thm:L}
	For any given $a>0$ and $b>1$ with $b>a$, there exits $L^*>\pi$ such that:
	\begin{itemize}
		\item  for all $0< L < L^*$, for any initial data $y_1^0, y_2^0 \in \mathcal{C}^0((0,L);[0,1])$, the system is asymptotically controllable to $(0,1)$.
		\item for all $L> L^*$, there exists a barrier, therefore, there exists some initial data $y_1^0, y_2^0 \in \mathcal{C}^0((0,L);[0,1])$ such that the system is not controllable to $(0,1)$ and $(w_1^*, w_2^*)$.
	\end{itemize} 
\end{conjecture}
The critical case $L=L^*$ is still unclear, since the method we used to establish the case of equality for $a^*$ and $b^*$ cannot be applied.

\begin{figure}[htbp]
    \centering
    \begin{subfigure}{0.48\textwidth}
        \centering
        \includegraphics[width=\linewidth]{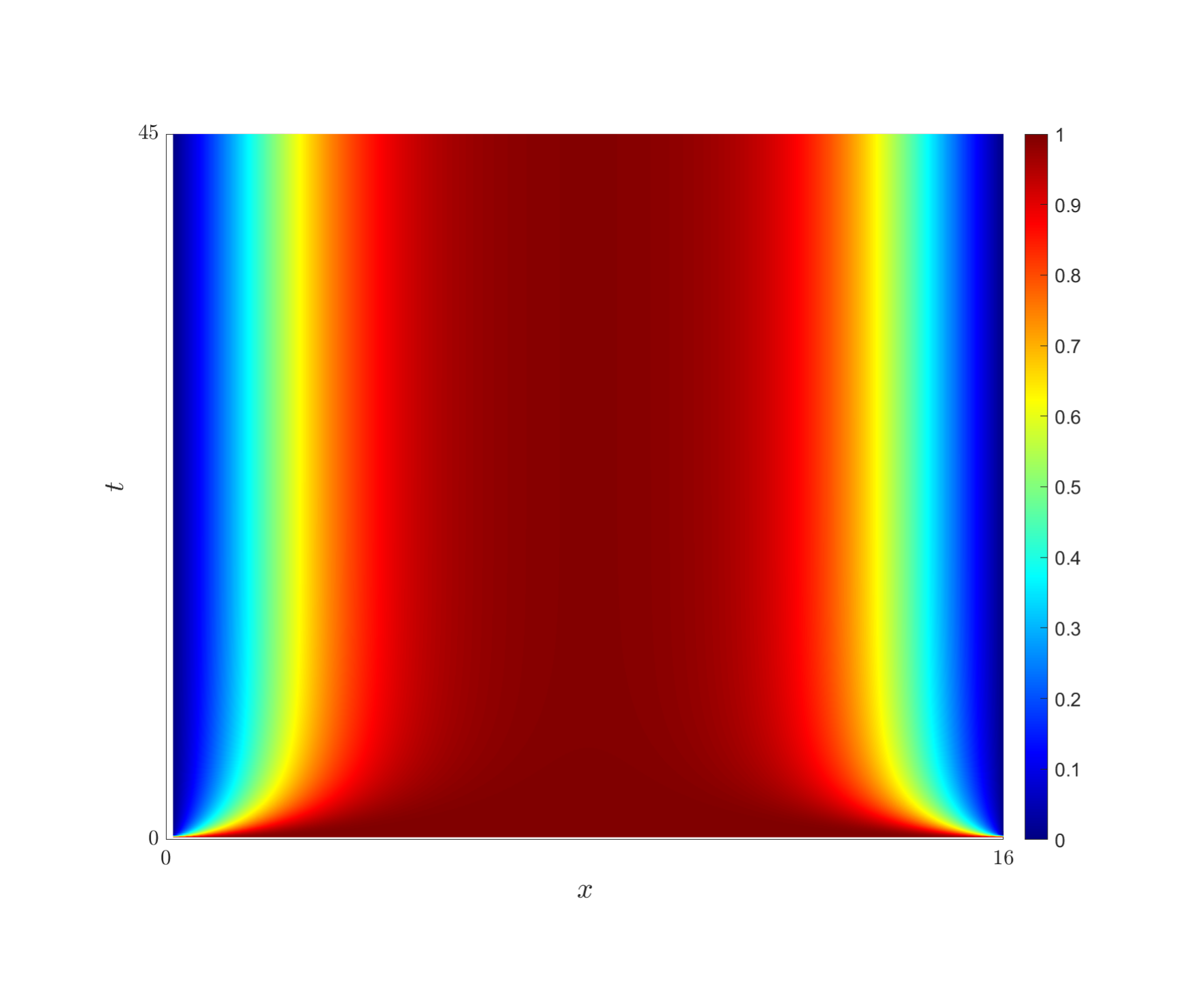}
        \caption{Solution $y_1$}
    \end{subfigure}
    \hfill
    \begin{subfigure}{0.48\textwidth}
        \centering
        \includegraphics[width=\linewidth]{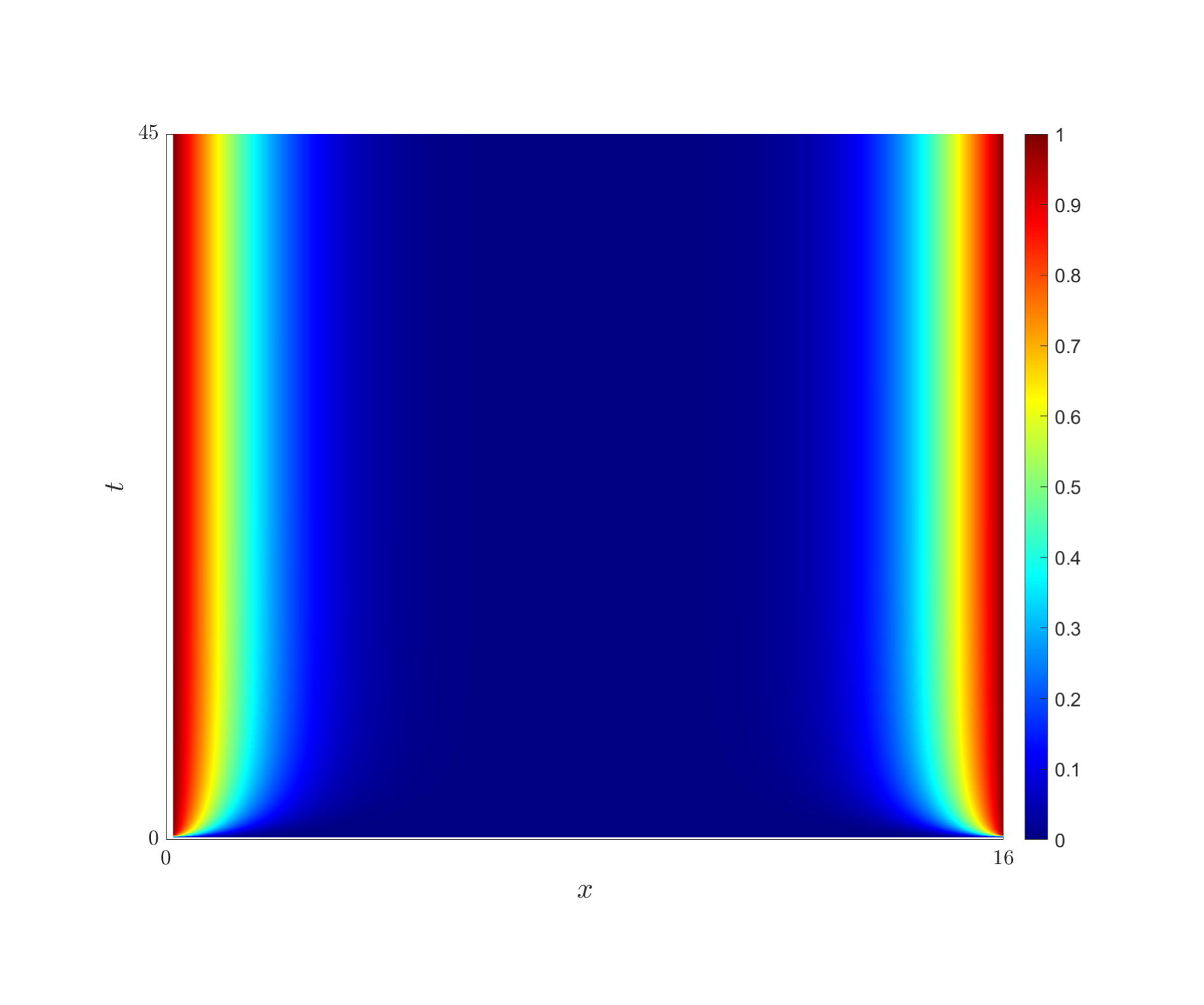}
        \caption{Solution $y_2$}
    \end{subfigure}
    \caption{The barrier effect: Controllability to the steady state $(0,1)$ from the initial state $(1,0)$ with $L=16$, $a=1.5$, $b=2.6$, and $T=45$, simulated in \textsc{Matlab} using the \texttt{CasADi} package. The target is not well approximated, and the system appears to converge to an steady state corresponding to the boundary data $(0,1)$, indicating the presence of a barrier.}
    \label{fig:L}
\end{figure}

\section{Proofs }

\subsection{Preliminaries}\label{ss:preliminars}

Here we collect some tools and results from the literature that we will use in the proofs of our theorems.

\medskip
\noindent
\textbf{Monotonicity structure.}
Monotone systems constitute a vast and well-studied area in the theory of partial differential equations. For a comprehensive treatment of monotone competitive systems, we refer the reader to \cite{Smith2008MonotoneDS}.

In such systems -- including \eqref{sys:lv} -- the dynamics preserve the partial ordering of the states. 
To formalize this idea, we define a \emph{subsolution} as a pair $(y_1, y_2)$ satisfying the following conditions:
\begin{equation}\label{def:subsol}
\begin{split}
\partial_t y_1 - \partial^2_x y_1 \leq y_1(1 - y_1 - a y_2), \qquad  & \text{for} \ x\in (0, L), \ t>0,\\
\partial_t y_2 -  \partial^2_x y_2 \geq   y_2(1 - b y_1 - y_2), \qquad  & \text{for} \ x\in (0, L), \ t>0.
\end{split}
\end{equation}
Similarly, a \emph{supersolution} is defined as a pair $(y_1, y_2)$ for which the corresponding inequalities hold in the opposite direction.

We also point out that if $(y_1,y_2)$ and $(z_1, z_2)$ are two pairs of subsolutions for the same interval $(0,L)$, then the pair $(w_1,w_2)$ defined by 
\begin{equation}\label{def:gensubsol}
    w_1(x)=\sup\{y_1(x),z_1(x)\}, \  w_2(x)=\inf \{ y_2(x), z_2(x) \}, \quad x\in (0,L),
\end{equation}
is called a \emph{generalised subsolution}. Similarly, 
if $(y_1,y_2)$ and $(z_1, z_2)$ are two pairs of supersolutions for the same interval $(0,L)$, then the pair $(w_1,w_2)$ defined by $$w_1(x)=\inf\{y_1(x),z_1(x)\}, \  w_2(x)=\sup \{ y_2(x), z_2(x) \}, \quad x\in (0,L),$$
is called a \emph{generalised supersolution}.

Moreover, the construction can be further generalised to cases where information on the two sub- (respectively super-) solutions is only local. Consider  $(x_1, x_2)\subseteq (0,L)$ an open interval and let $(y_1,y_2)$ and $(z_1, z_2)$ be two pairs of subsolutions defined on $(x_1, x_2)$. Suppose that 
$y_1(x_1) \geq z_1(x_1)$ and $y_1(x_2) \geq z_1(x_2)$, and that
$y_2(x_1) \leq z_2(x_1)$ and $y_2(x_2) \leq z_2(x_2)$.
Then, the pair $(w_1,w_2)$ defined by 
\begin{equation}\label{eq:gensub1}
    w_1(x)= 
    \begin{cases}
        y_1(x) & x \in (0, L)\setminus (x_1, x_2) \\
        \sup\{y_1(x),z_1(x)\} & x\in(x_1, x_2)
    \end{cases}
\end{equation}
and
\begin{equation}\label{eq:gensub2}
    w_2(x)= 
    \begin{cases}
        y_2(x) & x \in (0, L)\setminus (x_1, x_2) \\
        \inf \{ y_2(x), z_2(x) \} & x\in(x_1, x_2)
    \end{cases}
\end{equation}
is still a generalised subsolution. The analogous statement holds for supersolutions, with the roles of infimum and supremum interchanged.

We have the following comparison principle for Lotka -- Volterra systems (see for example \cite{smith1988systems} or \cite{girardin2019invasion}):

\begin{theorem}[Comparison Principle]\label{thm:cp}
	\emph{Let $I$ be an open interval of $\R$ and let $(\underline{y}_1, \underline{y}_2)$ and $(\overline{y}_1,\overline{y}_2)$ be respectively a (generalised) subsolution and a (generalised) supersolution of system \eqref{sys:lv} such that 
	\begin{equation*}
	\begin{split}
	\underline{y}_1(x,0) \leq \overline{y}_1(x,0) \quad &\text{and} \quad \overline{y}_2(x,0) \leq \underline{y}_2(x,0) \qquad \text{for all} \ x\in I, \\
	\underline{y}_1(x,t) \leq \overline{y}_1(x,t) \quad &\text{and} \quad \overline{y}_2(x,t) \leq \underline{y}_2(x,t) \qquad \text{for all} \ x\in \partial I, t\geq 0.	
	\end{split}
	\end{equation*}
	Then, 
	\begin{equation*}
	\underline{y}_1(x,t) \leq \overline{y}_1(x,t) \quad \text{and} \quad \overline{y}_2(x,t) \leq \underline{y}_2(x,t) \qquad \text{for all} \ x\in I, t\geq 0.
	\end{equation*}
	}
\end{theorem} 

This result says that the set of solutions has a natural ordering, and that the ordering of the parabolic boundary data is extended to the interior of the set.

Moreover, this result also holds in higher dimensions.

\medskip
\noindent
\textbf{Existence of constrained solutions for the boundary control problem.}
Notice that the constant solutions $(0,1)$ and $(1,0)$ act as sub- and super-solutions. Hence, the existence of solutions for initial data and controls respecting the bounds \eqref{hyp:bounds} is guaranteed by the classical theory of monotone methods (see for example \cite{sattinger1972monotone,cosner1984stable}). These solutions fulfil the same constraints.

\medskip
\noindent
\textbf{Travelling fronts.}
Travelling fronts are powerful tools in the analysis of Lotka -- Volterra equations. These are special solutions that connect two equilibria and can be viewed as time translations of a fixed spatial profile.

A \emph{travelling front solution} with speed $c\in\R$ connecting $(0,1)$ and $(1,0)$ is a pair
\begin{equation}\label{tf}
(\alpha,\beta)(x,t):=(Y_1, Y_2)(x-ct)
\end{equation}
where $Y_1, \ Y_2: \R\to \R$ solve the system
\begin{equation}\label{sys:tw}
\left\{
\begin{array}{l}
Y_1''+cY_1' + Y_1(1-Y_1-aY_2)=0, \\
Y_2''+cY_2' + Y_2(1-Y_2-bY_1)=0,
\end{array}\right.
\end{equation}
with the conditions
\begin{equation}\label{cond:tw}
(Y_1,Y_2)(-\infty)=(0,1), \quad (Y_1,Y_2)(+\infty)=(1,0), \quad Y_1'>0, \ Y_2'<0.
\end{equation}
Notice that a translation of a travelling front is still a travelling front.

Existence of travelling fronts connecting $(0,1)$ and $(1,0)$ and their uniqueness up to translations when $a,b >1$ was proved in \cite{gardner1982existence}, \cite{kan1996existence} and \cite{tang1980propagating}.
The sign of the speed was investigated in \cite{guo2013sign}.
In particular, the following holds:
\begin{theorem}[Theorem 1.1 in \cite{guo2013sign}]\label{thm:guosign}
	\emph{If $a<b$, then the speed $c$ of the travelling front defined by \eqref{tf}-\eqref{cond:tw} satisfies $c<0$. }
\end{theorem}

Therefore, as $t$ evolves from $-\infty$ to $+\infty$, the travelling wave solution passes from $(0,1)$ to $(1,0)$.

Note that by imposing alternative conditions in place of \eqref{cond:tw}, one could obtain different types of travelling fronts. A selection of such known fronts is illustrated in Figure \ref{fig:waves}.
However, throughout the proofs, we will only consider travelling fronts as in \eqref{cond:tw}. For simplicity, we will refer to these as travelling fronts without risk of ambiguity.

\medskip
\noindent
\textbf{Large time dynamics for the logistic equation.}
The following lemma,  widely known in the literature (see for example \cite{cantrell1989diffusive}), will be useful in the proof of several results in this paper, including Theorem \ref{thm:strong}.

\begin{lemma}\label{lemma:logistic}
	\emph{Let $\theta$ be the solution of the logistic equation  with Dirichlet boundary conditions as in \eqref{eq:logistic}.
	\begin{enumerate}
		\item Assume $L\leq \pi$. Then the only steady state solution to the logistic equation problem \eqref{eq:logistic} is the trivial one and any solution $\theta$ issued from a nonnegative initial datum satisfies $\theta \to 0$ as $t\to +\infty$.
		\item Assume $L>\pi$. Then \eqref{eq:logistic} admits a  steady state solution $\Theta$ which is strictly positive in $(0,L)$ and any solution $\theta$ issued from a nonnegative initial datum satisfies $\theta \to \Theta$ as $t\to +\infty$.
	\end{enumerate} 
	}
\end{lemma} 

\begin{remark}\label{rmk:multid}
This lemma extends naturally in dimension $n > 1$ to balls of radius $l$, provided that the value $\pi$ is replaced by the value $\ell=2r$, where $r\in \R^+$ is such that  the eigenvalue problem
$$ -\Delta u =\lambda u \ \  \text{in} \ B_r, \qquad u=0 \ \ \text{on} \  \partial B_r, \qquad u>0  \ \ \text{in} \ B_r,$$
admits the eigenvalue $\lambda =1$, $B_r$ being the ball of radius $r$. 
Then, all the results presented in Section \ref{s:main} that depend on the threshold $\pi$ may be adapted in higher dimensions by using the threshold $\ell$.    

Moreover, one can extend Lemma \ref{lemma:logistic} to the case of non-radial domains in higher dimensions. 
To do so, one has to consider a domain $\Omega \subset \R^n$, and study the eigenvalue problem 
$$ -\Delta u =\lambda u \ \  \text{in} \ \Omega, \qquad u=0 \ \ \text{on} \  \partial \Omega, \qquad u>0  \ \ \text{in} \ \Omega.$$
If $\lambda \geq 1$, then the first case of Lemma \ref{lemma:logistic} applies. If not, the second case applies. Then, the statements of Theorem \ref{thm:strong} and of the other results can be re-written by using $\lambda=1$ as a threshold.

\end{remark}

\medskip
\noindent
\textbf{The non-diffusive system.} 
As already mentioned in the introduction, 
ignoring diffusion and space structure, and in the absence of controls, system \eqref{sys:lv} boils down to the Lotka -- Volterra dynamical system \eqref{sys:odes}. 
The behaviour of the solutions of the ODE system \eqref{sys:odes} when $a>1$ is portrayed in Figure \ref{fig:odes}.
When the initial position $(w_1^0,w_2^0)$ is in the blue zone,  $(w_1(t),w_2(t))\to (1,0)$ as $t \to +\infty$; while, when it is in the green zone,  $(w_1(t),w_2(t))\to (0,1)$  as $t \to +\infty$. 
When, instead $a<1$, the equilibrium $(1,0)$ attracts  every trajectory  with initial conditions in $(0,1]\times[0,1]$.

\begin{center}
		\includegraphics[width=0.7\linewidth]{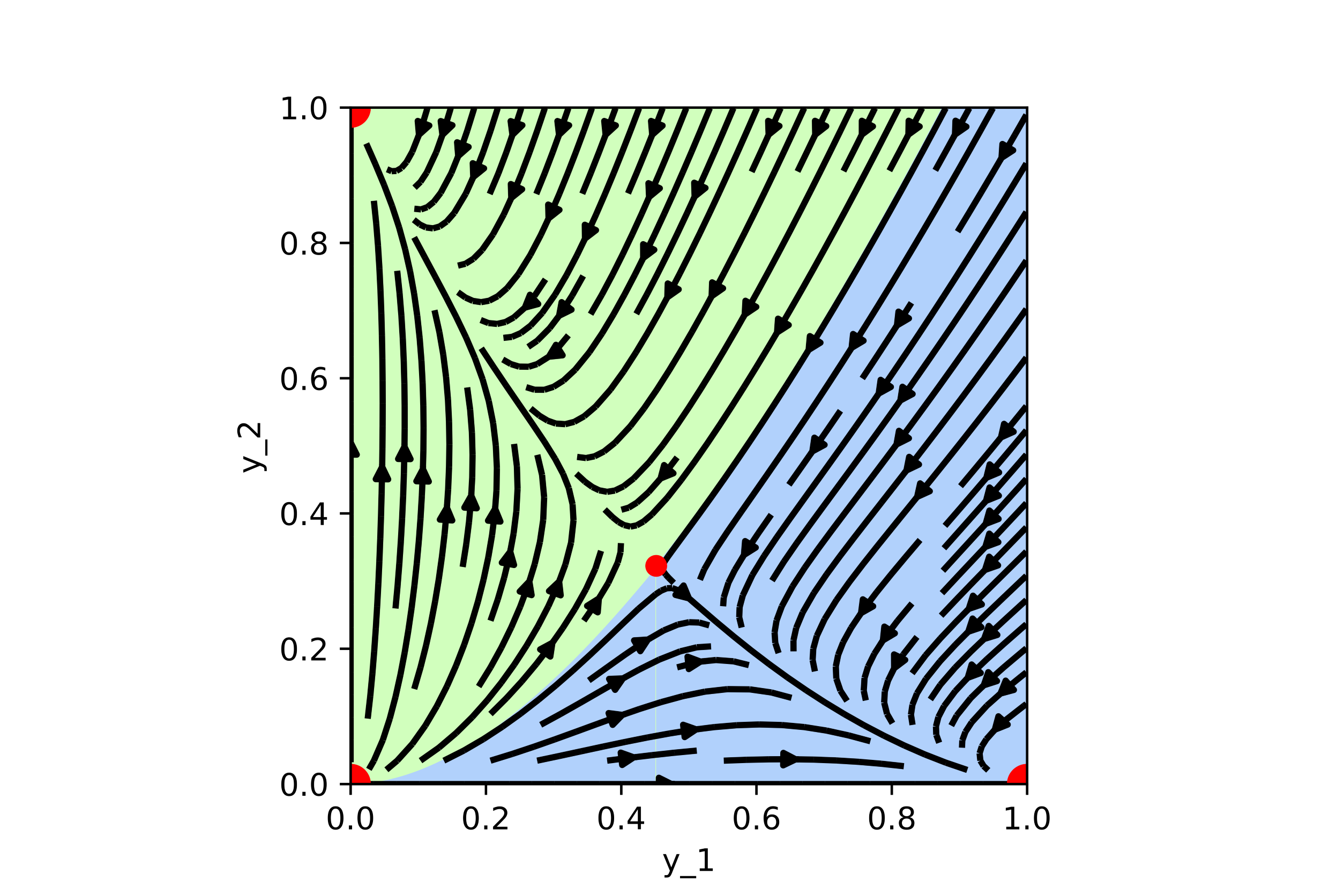}
	\captionof{figure}{\it Phase portrait of the non-diffusive model \eqref{sys:odes} with  $a=1.5$, $b=3.5$. In black, the orbits. The red dots represent the equilibria. The coloured zones represent the basin of attraction of the points: $(1,0)$ in light blue, and $(0,1)$ in light green.}
	\label{fig:odes}
\end{center}
\medskip
\noindent

When the coexistence one $(w_1^*, w_2^*)$ lies inside the square $[0,1] \times [0,1]$, it has a stable manifold that, for the system \eqref{sys:odes}, corresponds to the graph of the function
\begin{equation}\label{separatrix}h(y) = \frac{b - 1}{a - 1} y \qquad \text{for} \quad y \in \left[0, \frac{a - 1}{b - 1} \right].\end{equation}

To prove the forthcoming Theorem \ref{thm:implies}, we recall a result from \cite{iida1998diffusion} and establish a Lemma \ref{lemma:barrier} that relies upon it.
The result in \cite{iida1998diffusion} concerns the behavior of solutions to system \eqref{sys:lv} with Neumann boundary conditions, starting from initial data $(y_1^0, y_2^0)$ lying in the basin of attraction of the equilibrium $(0,1)$. Specifically, this means that for all $x \in (0, L)$, $(y_1^0(x), y_2^0(x))$ lies above the separatrix $h(y_1)$ defined in formula \eqref{separatrix}. The result was originally presented in general dimension $n\in\N^*$, but here we present it in dimension 1 for the sake of coherence of notation with respect to the rest of the paper.

\begin{theorem}[Theorem 1.1 in \cite{iida1998diffusion}]\label{thm:iida1}
		\emph{Let $a>1$ and $h(y)$ be as in \eqref{separatrix}.
	If $y_2^0(x)\geq h(y_1^0(x))$ and $y_2^0(x)\not\equiv h(y_1^0(x))$ in $(0,L)$, then the solution to \eqref{sys:lv} with Neumann boundary conditions
	\begin{equation}
		y_1'(0,t)=	y_2'(0,t) = y_1'(L,t)=	y_2'(L, t)=0 \qquad t>0,
	\end{equation}
	 satisfies $ \underset{t\to +\infty}{\lim} (y_1,y_2)(x,t)= (0,1)$ uniformly in $(0,L)$. 	}
\end{theorem}

\subsection{Proof of Theorem \ref{thm:strong}}
	We prove separately each statement.
	\medskip
	
	1.  Let us set
        \begin{equation}\label{1703}
            (u_1, u_2)(x,t)\equiv (1,0) \qquad \text{for} \  x=0, L, \quad t>0.
        \end{equation}
	Recall that $0$ is a subsolution of the first equation in system \eqref{sys:lv} and $1$ is a supersolution for the second equation.
	Hence, by the Maximum Principle for parabolic equations (see \cite{protter2012maximum}), we get that the solution $({y}_1,{y}_2)$ of system \eqref{sys:lv} with initial datum $(y_1^0, y_2^0)$ and boundary control $(u_1, u_2)$ at time $t=1$ satisfies
	\begin{equation*}
	m:=\underset{x\in [0,L]}{\min} \, {y}_1(x,1) >0, \quad M:=\underset{x\in [0,L]}{\max} \, {y}_2(x,1) <1,
	\end{equation*}
	Then, there exists a travelling front $(\alpha,\beta)$  as in formula \eqref{tf} such that 
    \begin{equation*}
        \alpha(x,1) < m \leq {y}_1(x,1), \quad  \beta(x,1)>M \geq {y}_2(x,1) \qquad \text{for} \ x \in (0,L).
    \end{equation*}
	Now, notice that
    \begin{equation*}
       \alpha(x,t) \leq  u_1(x,t), \quad \beta(x,t) \geq u_2(x,t) \qquad \text{for} \ x=0,L, \quad t>1.
    \end{equation*}
	Then, by the Comparison Principle in Theorem \ref{thm:cp},  we have that 
	\begin{equation}\label{1803}
	\alpha(x,t) \leq {y}_1(x,t), \ \beta(x,t) \geq {y}_2(x,t)  \qquad \text{for} \ x\in[0,L], \ t>1.
	\end{equation}
	
	From Theorem 1.1 in \cite{guo2013sign} (stated in this paper as Theorem \ref{thm:guosign}), we have that the speed $c$ of the travelling front is positive, so 
	\begin{equation}\label{1854}
	(\alpha,\beta) \overset{t\to +\infty}{\longrightarrow} (1,0).
	\end{equation}
	By applying the Comparison Principle in Theorem \ref{thm:cp}, using \eqref{1803} and \eqref{1854}, we have that $({y}_1,{y}_2)$ with boundary controls as in \eqref{1703}  satisfies
    \begin{equation*}
	({y}_1,{y}_2) \overset{t\to +\infty}{\longrightarrow} (1,0).
	\end{equation*}
    
	\medskip

	2. Let $L\leq \pi$ and $y(x,t)$ be a solution to the logistic  Dirichlet problem \eqref{eq:logistic} with initial data $y_1^0\in L^{\infty}(0,L)$. By Lemma \ref{lemma:logistic}, the only steady state solution to \eqref{eq:logistic} is the trivial one and $y$ must converge to a steady state solution as $t\to +\infty$. Then,
	\begin{equation}\label{1653}
		y(x,t)\to 0 \qquad \text{as} \ t\to +\infty.
	\end{equation}
	
	For all $(y_1,y_2)$ solution to \eqref{sys:lv} with initial datum $(y_1^0,y_2^0)$ and boundary conditions $(u_1,u_2)=(0,1)$, we have
	\begin{equation*}
	\partial_t y_1-\partial_{xx}^2 y_1 \leq y_1(1-y_1)\ \text{in} \ (0,L), \ \ y_1(0)=y_1(L)=0, \ \ y_1(x,0)=y_1^0 (x)\ \text{in} \ (0,L),
	\end{equation*}
	so that $y_1$ is a subsolution for the logistic equation problem \eqref{eq:logistic}.
	Then, by the Comparison Principle for parabolic equations, we have $y_1(x,t) \leq y(x,t)$ for all $t\geq 0$. In particular, by \eqref{1653}, 
	\begin{equation}\label{1206}
		y_1(x,t)\to 0 \quad \text{as } t\to +\infty.
	\end{equation}

	By the main result of \cite{smillie1984competitive}, every bounded solution of \eqref{sys:lv} converges to a steady state. 
	By \eqref{1206}, we have that $y_2$ converges to the solution $\zeta(x)$ of the steady state problem
	\begin{equation*}
	- \zeta'' = \zeta(1-\zeta) \quad \text{in} \ (0,L), \qquad \zeta(0)=\zeta(L)=1.
	\end{equation*}
	Let us perform the change of variables $z(x):=1-\zeta(x)$. Then, $z$ solves 
	\begin{equation*}
	- z'' = -z(1-z) \quad \text{in} \ (0,L), \qquad z(0)=z(L)=0.
	\end{equation*}
	By multiplying both sides of the latter equation by $z$ and integrating in $(0,L)$, we get
	\begin{equation*}
	0 \leq \left[-z'(x) z(x)\right]\big|_0^L+ \int_{0}^{L} (z'(x))^2 dx = \int_{0}^{L} - z^2(x)(1-z(x)) dx \leq 0. 
	\end{equation*}
	The latter is true only if $z(x)\equiv 0$, which corresponds to $\zeta(x)\equiv 1$. Hence, $(y_1,y_2)$ converges to $(0,1)$ with the boundary control $(u_1,u_2)=(0,1)$, as desired.
	
	\medskip
	
	3. Consider the solution $(z_1,z_2)$ of the diffusive system \eqref{sys:lv} with initial datum $(y_1^0(x),y_2^0(x)) \in B(0,1)$ for all $x\in(0,L)$ and Neumann boundary conditions
	\begin{equation*}
	{\partial_{x} z_1}(x,t)= {\partial_{x} z_2}(x,t)=0, \qquad \text{for} \ x=0,L, \ t>0.
	\end{equation*}
	Notice that we still have $0 \leq z_1, z_2 \leq 1$, thanks to the Comparison Principle.
	By Theorem 1.1 in \cite{iida1998diffusion}, we also have that
	\begin{equation}\label{1257}
	(z_1,z_2) \overset{t\to +\infty}{\longrightarrow} (0,1).
	\end{equation}
	Then, $(z_1,z_2)$ is a solution of the diffusive system \eqref{sys:lv} with initial datum $(y_1^0,y_2^0)$ and boundary control 
	\begin{equation*}
	(\tilde{u}_1, \tilde{u}_2)(x,t) := (z_1,z_2)(x,t) \qquad \text{for} \ x=0,L, \ t>0,
	\end{equation*}
	and it has the desired asymptotic behaviour. Now consider $(y_1, y_2)$ solution of the diffusive system \eqref{sys:lv} with initial datum $(y_1^0,y_2^0)$ and boundary control $(u_1,u_2)(x,t)=(0,1)$ for $x=0, \ L$ and $t>0$. Since $u_1 \leq \tilde{u}_1$ and $u_2 \geq \tilde{u}_2$,  for $x=0, \ L$ and $t>0$, and the initial datum coincides, by the Comparison Principle we have that $y_1 \leq z_1$ and $y_2 \geq z_2$. By this and \eqref{1257}, we get the desired result. 
	
	\medskip
	
	4. 
	The argument for the convergence of $y_1$ to 0 presented in the proof of point 2 can be repeated for both $y_1$ and $y_2$. 
	\medskip

	5. Assume, by contradiction, that there exists a control $(u_1, u_2) \in C(\{0, L\} \times (0, +\infty))$ such that the corresponding solution $(y_1, y_2)$ of system \eqref{sys:lv} satisfies
	\begin{equation}\label{2030}
	(y_1, y_2) \overset{t\to\infty}{\longrightarrow} (0,0).
	\end{equation}
	 Let us call $\zeta :=y_1+y_2$, then $\zeta$ satisfies
	\begin{align*}
	\partial_t \zeta - \partial_{xx}^2 \zeta = \zeta(1-\zeta) -(a+b-2)y_1y_2 \ \ &\text{on} \ (0,L), \\ \zeta= u_1+u_2 \ \ &\text{on} \ \{0,L\}\times [0,+\infty).
	\end{align*}
	If $a+b-2 \leq 0$, then
	\begin{equation}\label{1217}
		\partial_t \zeta - \partial_{xx}^2 \zeta \geq \zeta(1-\zeta).
	\end{equation}
    So, $\zeta$ is a supersolution for the logistic equation \eqref{eq:logistic} with $\theta(x,0)=y_1^0(x)+y_2^0(x)$.
	Notice that, for $L>\pi$, by Lemma \ref{lemma:logistic}, \eqref{eq:logistic}  has a positive attractive steady state solution $\Theta$. 
	Being $\zeta$ a supersolution for \eqref{eq:logistic} and by \eqref{2030}, we have that
	\begin{equation*}
	0 = \underset{t\to+\infty}{\lim}  \zeta(x,t) \geq \Theta(x) \quad \text{in} \ (0,L),
	\end{equation*}
	which is impossible. Hence, it is not possible to control any initial datum to $(0,0)$ if $a+b-2 \leq 0$.
    
	Let us now show the same for $a+b-2>0$. Since $$ y_1 y_2 \leq \frac{(y_1+y_2)^2}{4},$$
	the function $\zeta$ satisfies 
	\begin{equation*}
	\partial_t \zeta - \partial_{xx}^2 \zeta \geq \zeta \left( 1- \frac{a+b+2}{4} \zeta  \right)\quad \text{on} \ (0,L), \quad \zeta(0)\geq 0, \ \zeta(L) \geq 0.
	\end{equation*}
	So, the function
	\begin{equation*}
	\sigma (x, t) := \frac{a+b+2}{4} \zeta (x,t)
	\end{equation*}
	is a supersolution for the logistic equation \eqref{eq:logistic} with $$\theta(x,0)=\dfrac{a+b+2}{4} (y_1^0(x)+y_2^0(x)).$$
	Notice that, for $L>\pi$, by Lemma \ref{lemma:logistic}, \eqref{eq:logistic}  has a positive attractive steady state solution $\Theta$. 
	Being $\sigma$ a supersolution for \eqref{eq:logistic} and by \eqref{2030}, we have that
	\begin{equation*}
	0 =  \frac{a+b+2}{4} \underset{t\to+\infty}{\lim}  \zeta(x,t) \geq \Theta(x) \quad \text{in} \ (0,L),
	\end{equation*}
	which is impossible. Therefore, even in the case $a + b - 2 > 0$, it is not possible to steer an arbitrary initial datum to $(0,0)$. This concludes the proof of the theorem.

\subsection{Properties of the barriers}

In this subsection, we present several useful results. In particular, we establish that the existence of barriers implies non-controllability, demonstrate the monotonicity of barriers, and, for $a > 1$, prove their continuous dependence on the system parameters.

\subsubsection{Existence of barriers implies non controllability}

In this subsection, we prove Theorem \ref{thm:implies}, which will be instrumental in the proofs of both Theorems \ref{thm:b} and \ref{thm:a}.
The theorem claims that the existence of barriers implies the non-controllability to the steady state $(0,1)$ and to the coexistence steady state $(w_1^*, w_2^*)$, whenever it exists.

We first state and prove the following technical lemma:

\begin{lemma}\label{lemma:barrier}
	\emph{If $a > 1$ (i.e., if $(w_1^*,w_2^*)$ is  in the range $[0,1]\times[0,1]$), then any solution $(\phi, \psi)$ to system \eqref{sys:barrier} satisfies the following property: there exists a point $x^* \in (0, L)$ such that at least one of the following inequalities holds:
\begin{equation*}
\phi(x^*)> w_1^* \quad \text{or} \quad \psi(x^*)< w_2^*.
\end{equation*}
}
\end{lemma}

This lemma indicates that the barrier cannot lie entirely below the threshold defined by the coexistence equilibrium.

Let us now prove Lemma \ref{lemma:barrier}. We will make use of the Theorem \ref{thm:iida1} from \cite{iida1998diffusion}, which we recalled in the Preliminaries Subsection.

\begin{proof}[Proof of Lemma \ref{lemma:barrier}]
	Let us argue by contradiction and assume that the statement is false. That is, assume that, for all $x \in (0, L)$, the following holds:
	\begin{equation*}
	\phi(x) \leq w_1^* \quad \text{and} \quad \psi(x)\geq w_2^*.
	\end{equation*}
    Then, it holds that $\psi(x)\geq w_2^* = h(w_1^*) \geq h(\phi(x))$ for all $x\in (0,L)$, with $h$ as in \eqref{separatrix}. Moreover, $1 =\psi(0) > h(\phi(0))=0$, so $\psi(x) \not \equiv h(\phi(x))$ for $x\in (0,L)$.
	Then, by Theorem \ref{thm:iida1} (Theorem 1.1 in \cite{iida1998diffusion}), the solution $(y_1, y_2)$ of system \eqref{sys:lv} with Neumann boundary conditions and initial data $(\phi, \psi)$ satisfies
	\begin{equation}\label{1439}
	(y_1, y_2)(x,t) \overset{t\to+\infty}{\longrightarrow} (0,1),
	\end{equation}	
	uniformly in $(0,L)$.
	
	We take the couple of time-dependent controls $(u_1, u_2)$ as the trace of the solution $(y_1,y_2)$ to the boundary $x=0$ and $x=L$. Thanks to the bounds \eqref{hyp:bounds}, it holds
	\begin{equation*}
	u_1(x,t) \geq 0, \ u_2(x,t) \leq 1, \qquad \text{for} \ x=0,L, \ t>0.
	\end{equation*}
	Then, by the Comparison Principle, it follows that 
	\begin{equation*}
	y_1(x,t) \geq \phi(x), \ y_2(x,t) \leq \psi(x), \qquad \text{for all} \ x\in(0,L), \ t>0,
	\end{equation*}	
	which is in contradiction with \eqref{1439}.
This completes the proof of the lemma.
	
\end{proof}

With Lemma \ref{lemma:barrier}, we are able to prove Theorem \ref{thm:implies}:

\begin{proof}[Proof of Theorem \ref{thm:implies}]
	Let $(\phi, \psi)$ be a barrier, i.e., a solution to system \eqref{sys:barrier}, for the parameters $\bar{a}$, $\bar{b}$ and $\bar{L}$.
	Let us consider an initial datum $(y_1^0,y_2^0)\in \mathcal{I}_{\phi,\psi}$, that is,  such that $y_1^0(x) \geq \phi(x)$ and $y_2^0 (x)\leq \psi(x)$ for all $x\in(0,\bar{L})$. 
	Let $(y_1, y_2)$ be the solution issued from $(y_1^0, y_2^0)$ with any boundary controls $(u_1, u_2)$.	
	Notice that, by the hypothesis on the problem stated in \eqref{hyp:bounds}, we have that
	\begin{align*}
		u_1(x,t) &\geq 0, \quad \text{for all} \ x \in (0,\bar{L}), \ t>0, \\
		u_2(x,t) &\leq 1, \quad \text{for all} \ x \in (0,\bar{L}), \ t>0.
	\end{align*}
	
	Thus, by applying the Comparison Principle as stated in Theorem  \ref{thm:cp} to $(y_1,y_2)$ and $(\phi, \psi)$, we get that
	\begin{align*}
		y_1(x,t) &\geq \phi(x) > 0, \quad \text{for all} \ x \in (0,\bar{L}), \ t>0, \\
		y_2(x,t) &\leq \psi(x) < 1, \quad \text{for all} \ x \in (0,\bar{L}), \ t>0.
	\end{align*}
	In particular,
$(y_1, y_2)$ cannot converge to $(0,1)$ as $t\to +\infty$.

	Moreover, if $\bar{a}>1$, by Lemma \ref{lemma:barrier}, there exists $x^*\in(0, \bar{L})$ such that either $\phi(x^*)>w_1^*$ or $\psi(x^*)<w_2^*$. Let us suppose  that the latter is true, the other case being analogous. Then, again by the Comparison Principle of Theorem \ref{thm:cp}, we have that
	$$ y_2(x^*,t) \leq \psi(x^*) < w_2^*, \quad \text{for all} \ t>0. $$
	Thus,
		$(y_1, y_2)$ cannot converge to $(w_1^*,w_2^*) $ as $t\to +\infty.$
		\end{proof}

\subsubsection{Monotonicity with respect to $a$, $b$, and $L$}

The existence of a barrier exhibits monotonicity with respect to the parameters $L$ and $b$, as stated in the following proposition.
 This will be crucial to prove the threshold phenomena in Theorems \ref{thm:b} and \ref{thm:a}.

\begin{proposition}\label{prop:mono}
	\emph{Suppose there exists a nontrivial solution to system \eqref{sys:barrier} for some parameters $\bar{a}$, $\bar{b}$, $\bar{L}$.
	The following holds:
	\begin{enumerate}
		\item For all $L\geq \bar{L}$ there exists a solution to system \eqref{sys:barrier} for the parameters $\bar{a}$, $\bar{b}$, and $L$.
		\item For all $b\geq \bar{b}$, there exists a solution to system \eqref{sys:barrier} for the parameters $\bar{a}$, $\bar{L}$, and $b$.
		\item For all $a\leq \bar{a}$, there exists a solution to system \eqref{sys:barrier} for the parameters $a$, $\bar{L}$, and $\bar{b}$.
	\end{enumerate}	}
\end{proposition}

\begin{proof}
	Suppose there are some values $\bar{a}>0$, $\bar{b}>1$ and $\bar{L}>\pi$ for which there exists a nontrivial solution $(\bar{\phi}, \bar{\psi})$ to system \eqref{sys:barrier}.
	
	1. Let us take some $L>\bar{L}$ and consider the pair $(\sigma, \omega)$ on $(0,L)$ defined as
	\begin{equation*}\label{def:phi}
	\sigma(x)=
	\left\{
	\begin{array}{ll}
	\bar{\phi}(x), & \text{for} \ x\in(0,\bar{L}), \\
	0,  & \text{for} \ x\in(\bar{L},L),
	\end{array}\right.
	\end{equation*}
	and
	\begin{equation*}\label{def:psi}
	\omega(x)=
	\left\{
	\begin{array}{ll}
	\bar{\psi}(x), & \text{for} \ x\in(0,\bar{L}), \\
	1,  & \text{for} \ x\in(\bar{L},L).
	\end{array}\right.
	\end{equation*}
	We now show that the couple $(\sigma, \omega)$ is a subsolution to system \eqref{sys:barrier} for the parameters $\bar{a}$, $\bar{b}$ and $L$. First of all,  $\sigma$ and $\omega$ satisfy the boundary conditions of \eqref{sys:barrier}. Then, $\sigma$ and $\omega$ are continuous since $\bar{\phi}(\bar{L})=0$ and $\bar{\psi}(\bar{L})=1$. Moreover,
	\begin{equation*}
	\sigma(x)=\max \{0, \bar{\phi}(x)\}, \ \omega(x)=\min \{1, \bar{\psi}(x)\}, \qquad \text{for} \ x \in (0, \bar{L}), 
	\end{equation*} 
	so that $\sigma$ is a generalised subsolution for the first equation and $\omega$ is a generalised supersolution for the second equation for system \eqref{sys:lv}.
	Then, by  \eqref{eq:gensub1} and \eqref{eq:gensub2}, $(\sigma, \omega)$ is a generalised subsolution for system \eqref{sys:lv}.
	
	So, the solution $(y_1, y_2)$ of system \eqref{sys:lv} issued from the initial datum $(\sigma, \omega)$ and with boundary conditions $(u_1,u_2)\equiv (0,1)$ must satisfy,  by the Comparison Principle, $ y_1(x,t)\geq \sigma(x)$, $ y_2(x,t)\leq \omega(x)$ for all $x\in(0,L)$ and $t\geq 0$.
Then, by the classical result of \cite{smillie1984competitive}, it follows that $(y_1, y_2)$ converges to a nontrivial steady state $(\phi, \psi)$, which is a solution of system \eqref{sys:barrier}.
	
	\medskip
	
	2. Let us take $b>\bar{b}$. We observe that 
	\begin{equation*}
	\left\{
	\begin{array}{ll}
	-  {\bar{\phi}}'' = \bar{\phi}(1 - \bar{\phi} - \bar{a} \bar{\psi}), & \text{for} \ x\in(0,\bar{L}), \\
	- {\bar{\psi}}'' \geq  \bar{\psi}(1 - b \bar{\phi} - \bar{\psi}), & \text{for} \ x\in(0,\bar{L}), \\
	\bar{\phi}(x)=0, \ \bar{\psi}(x)=1,  & \text{for} \ x=0, \bar{L}.
	\end{array}\right.
	\end{equation*}
	Hence, $(\bar{\phi},\bar{\psi})$ is a subsolution of \eqref{sys:barrier} with parameters $\bar{a}$ and $b$. 
	
	Then, the solution $(y_1, y_2)$ of system \eqref{sys:lv} issued from the initial datum $(\bar{\phi},\bar{\psi})$ and with boundary conditions $(u_1,u_2)\equiv (0,1)$ satisfies $ y_1(x,t)\geq \bar{\phi}(x)$, $ y_2(x,t)\leq \bar{\psi}(x)$ by the Comparison Principle.
	As before, thanks to \cite{smillie1984competitive}, $(y_1,y_2)$ converges to a nontrivial solution $(\phi,\psi)$ of system \eqref{sys:barrier}.
	
	\medskip
	
	3. We observe that for $a<\bar{a}$ we get
	\begin{equation*}
	\left\{
	\begin{array}{ll}
	-  {\bar{\phi}}'' = \bar{\phi}(1 - \bar{\phi} - \bar{a} \bar{\psi}) \leq \bar{\phi}(1 - \bar{\phi} - a \bar{\psi}) , & \text{for} \ x\in(0,\bar{L}), \\
	- {\bar{\psi}}'' =  \bar{\psi}(1 - \bar{b} \bar{\phi} - \bar{\psi}), & \text{for} \ x\in(0,\bar{L}), \\
	\bar{\phi}(x)=0, \ \bar{\psi}(x)=1,  & \text{for} \ x=0, \bar{L}.
	\end{array}\right.
	\end{equation*}
We apply the same argument as in the previous step to conclude the proof of the lemma.
 
\end{proof}

\subsubsection{Continuous dependence on $a$ and $b$}

In the following proposition, we show that the existence of a barrier depends  on $a$ and $b$ in a continuous way when $a>1$. This result plays an important role in determining the limit cases $a^*$ and $b^*$ in Theorems \ref{thm:a} and \ref{thm:b}.

\begin{proposition}\label{prop:dependence}
  \emph{  Let us consider a sequence $\{ (a_n,b_n) \}_{n\in\N}$ such that each couple satisfies \eqref{hyp:sc} and let us suppose that system \eqref{sys:barrier} has a solution for the parameters $a_n$ and $b_n$ for all $n\in \N$. Let us also suppose that $(a_n, b_n)\to (a,b)$ with $a$ and $b$ respecting \ref{hyp:sc} and $a > 1$. Then, system \eqref{sys:barrier} has a solution for the parameters $a$ and $b$.}
\end{proposition}

\begin{proof}
    First of all, since $a>1$, up to a subsequence, we can suppose 
    \begin{equation}\label{1229}
        a_n \in (a-1, a+1)\cap (1, +\infty), \qquad b_n\in (b-1,b+1).
    \end{equation}
    Let us denote by $(\phi_{n}, \psi_{n})$ the solution of system \eqref{sys:barrier} for the parameters $a_n$ and $b_n$. 
    Then, by the last line of \eqref{sys:barrier} we have that 
\begin{equation*}
    \underset{[0,L]}{\sup} \phi_n , \  \underset{[0,L]}{\sup}  \psi_n  \leq 1
\end{equation*}
    for all $n\in \N$.
 Then, 
 \begin{equation}\label{1248}
    \underset{(0,L)}{\sup} \phi_n'' , \  \underset{(0,L)}{\sup}  \psi_n''  \leq 3+a+b
\end{equation}
thanks to \eqref{1229}. 
Then $(\phi_n, \psi_n)$ is compact in $\mathcal{C}^{1,\alpha}(0,L)$, that is, up to a subsequence there exists $(\phi, \psi)$ such that $(\phi_n, \psi_n)\to (\phi, \psi)$ in $\mathcal{C}^{1,\alpha}$. In particular, passing to the limit in system \eqref{sys:barrier}, we get that $\phi(0)=\phi(L)=0$ and $\psi(0)=\psi(L)=1$. Moreover, $\phi$ and $\psi$ satisfy the equations in \eqref{sys:barrier} in a weak sense. Then, by Schauder estimates, we get that $\phi$ and $\psi$ satisfy the equations in \eqref{sys:barrier} in a classical sense.

Now, we show that $(\phi,\psi)\neq (0,1)$. Let us define 
\begin{equation*}
     \left(w_{n,1}^*, w_{n,2}^* \right) = \left( \frac{a_n-1}{a_n b_n -1}, \frac{b_n-1}{a_n b_n -1}\right).
\end{equation*}
Thanks to \eqref{1229}, by Lemma \ref{lemma:barrier} we get that for all $n\in\N$ there exists $x_n\in (0,L)$ such that either $\phi_n(x_n) > w_{n,1}^*$ or $\psi_n(x_n) < w_{n,2}^*$. Let us suppose that the first case happens for infinitely many different values of $n$. Then, 
\begin{equation*}
    \underset{n\to +\infty}{\lim} \underset{(0,L)}{\sup} \phi_n \geq \frac{a-1}{ab-1} >0.
\end{equation*}
This, together with \eqref{1248}, implies that $\phi_n$ cannot converge to $0$. Consequently, $\psi_n$ cannot converge to $1$.

If the inequality $ \phi_n(x_n) > w_{n,1}^*$ fails to hold for infinitely many $ n \in \mathbb{N}$, then  $\psi_n(x_n) < w_{n,2}^*$ holds for infinitely many different values of $n$. In this case, the same argument applies, with the supremum replaced by an infimum.

Thus, we have established the existence of a pair $(\phi, \psi)$ satisfying system \eqref{sys:barrier} for the given parameters $a$ and $b$.

\end{proof}

\subsection{Proof of Theorem \ref{thm:b}}

To prove Theorem \ref{thm:b}, we will show the following technical result, which is the core of the proof.

\begin{proposition}\label{thm:bbarrier}
	\emph{For any given $a>0$ and $L>\pi$, there exits $\bar{b}>1$ such that  there exists a barrier, that is, a nontrivial solution to system \eqref{sys:barrier} with parameters $a$, $\bar{b}$ and $L$. }
\end{proposition}

To prove Proposition \ref{thm:bbarrier}, we need a technical lemma:

\begin{lemma}\label{lemma:psi}
	\emph{Let us consider $R>\pi$, $\varepsilon\in(0,1)$, $C>0$, and $\phi(x)\in C^0(0,R)$ such that $\phi(0)=\phi(R)=0$, $\phi >0$ in $(0,R)$. Then, there exits $\bar{b}>1$ such that the solution of the problem
	\begin{equation}\label{1113}
		-\psi '' = \psi (1-\bar{b}\phi -\psi) \ \text{in} \ (0,R),  \quad \psi(0)=\psi(R)= 1-\varepsilon,
	\end{equation}
	satisfies 
	\begin{equation*}
		\psi < 1-\varepsilon \ \text{in} \ (0,R), \quad \psi'(0)< - C, \quad \psi'(R)>C.
	\end{equation*}}
\end{lemma}

\begin{proof}

    Let us repeat that the existence of a solution to the Dirichlet problem \eqref{1113} is guaranteed by the classical theory of monotone methods (\cite{sattinger1972monotone}).
    
	We divide the proof in two steps.	
	
	\emph{Step 1. For all $0<\theta<R/3$ and $\delta>0$, there exists $\bar{b}$ such that $\psi(x)<\delta$ in $(3\theta,R-3\theta)$.}
	
	Suppose by contradiction that this is not true, so
	\begin{equation}\label{1629}
		\underset{(3\theta,R-3\theta)}{\sup} \psi(x)>\delta,
	\end{equation}
	for some  $\theta>0$ and $\delta>0$ and for all $b>0$. 
	Let us choose $\bar{b}$ large enough to have
	\begin{equation}\label{2102}
		\max \left\{1, \frac{2}{\delta\theta^2}+1 \right\}<\bar{b}\phi(x) \quad \text{in }(\theta,R-\theta).
	\end{equation}
	Hence, we have
	\begin{equation}\label{1254}
		-\psi''(x)<0 \quad \text{in }(\theta,R-\theta).
	\end{equation}
	So, by the maximum principle, for all $(x_1,x_2)\subset(\theta,R-\theta)$ we get that 
	\begin{equation}\label{1645}
	\underset{(x_1,x_2)}{\max} \psi(x)= \max \{\psi(x_1), \psi(x_2)\}.
	\end{equation}
	Let us now suppose that $\psi(3\theta)\geq \psi(R-3\theta)$, the other case being identical.
	Hence, by \eqref{1629} and \eqref{1645} we have that 
	$$ \psi(3\theta)>\delta. $$
	Then, by \eqref{1645} we also have that $\psi$ is nonincreasing in $(\theta, 3\theta)$, so $\psi'\leq 0$ in $(\theta, 3\theta)$ and
	\begin{equation}\label{1249}
		\psi(x)>\delta \quad \text{in } (\theta, 3\theta)
	\end{equation} 
	Now, we have that
	\begin{align*}
	0\geq \psi(3\theta)- \psi(2\theta) &=   \int_{2\theta}^{3\theta} \psi'(x) dx, \\
	& = \int_{{2\theta}}^{3\theta} \left( \psi'(2\theta)+ \int_{{2\theta}}^{x} \psi''(y)dy \right) dx, \\
	& = \int_{{2\theta}}^{3\theta} \left( \psi'(2\theta)+ \int_{{2\theta}}^{x} \psi (b\phi-1+\delta) dy \right) dx, \\
	& \geq\psi'(2\theta)\theta + 1,
	\end{align*}
	thanks to \eqref{1249} and the choice of $b$ in \eqref{2102}.
	Thus, we get that
	$$ \psi'(2\theta) \leq -\frac{1}{\theta}.$$ 
	Also, by \eqref{1254} we get that $\psi'(x)<\psi'(2\theta)$ for all $x\in(\theta, 2\theta)$. Hence, 
	\begin{equation*}
		\psi(\theta)= \psi(2\theta) + \int_{2\theta}^\theta \psi'(x)dx > 1, 
	\end{equation*}
	which is absurd. Then, the claim is proved.
	
	\emph{Step 2. For such $\bar{b}$, $\psi'(0)<-C$, $\psi'(R)>C$ and $\psi(x)<1-\varepsilon$ for $x\in(0,R)$.}
	
	Let us take $\delta\in(0, 1-\varepsilon)$ and 	
	\begin{equation}\label{1612}
		0 < \theta < \min \left\{ \frac{C}{3}, \frac{1-\varepsilon-\delta}{6C}  \right\}.
	\end{equation}
	By applying the previous step, we get that there exists $\bar{b}$ such that
	\begin{equation}\label{1713}
		\psi(x)<\delta \quad \text{for } x\in(3\theta, R-3\theta).
	\end{equation} 
	Then, by Lagrange's Theorem, there exist $\bar{x}\in (0,3\theta)$ such that
	\begin{equation*}
		\psi(\bar{x}) \leq \frac{\delta-(1-\varepsilon)}{3\theta} < -2C,
	\end{equation*}
	 where the last inequality is obtained thanks to \eqref{1612}.
	 
	 Now, we have that
	 \begin{align*}
	 	\psi'(\bar{x})-\psi'(0) &=\int_0^{\bar{x}} \psi''(x)dx \\
	 	& = \int_0^{\bar{x}} -\psi(x)(1-b\phi(x)-\psi(x))  dx \\
	 	& \geq -\bar{x} \geq -3\theta.
	 \end{align*}
	 Then, again by \eqref{1612}, we get
	 \begin{equation*}
	 	\psi'(0) < -2C+3\theta < -C.
	 \end{equation*}
	 Moreover, with the same argument, we can show that
	 \begin{equation*}
	 \psi'(x) < 0 \quad \text{in} \ (0, 3\theta).
	 \end{equation*}
	 Hence, we get that
	 \begin{equation*}
	 \psi(x) < 1-\varepsilon \quad \text{in} \ (0, 3\theta).
	 \end{equation*}
	 With a similar argument we can prove that $\psi'(R)>C$ and 
	 \begin{equation*}
	 \psi(x) < 1-\varepsilon \quad \text{in} \ (R- 3\theta, R).
	 \end{equation*}
	 Together with \eqref{1713}, this gives us $\psi(x)<1-\varepsilon$ for $x\in(0,R)$, as wished.		
\end{proof}

We underline that this lemma can also be extended in higher dimensions at least in the case of the domain being a ball or on a star domain; in fact, it relies principally on the maximum principle, and on a computation along paths connecting some interior points to the boundary.

We are now ready to give the proof of Proposition \ref{thm:bbarrier}.
\begin{proof}[Proof of Proposition \ref{thm:bbarrier}]
	Let us fix 
	\begin{equation}\label{1927}
		\varepsilon\in \left(\frac{a-1}{a}, 1  \right) \cap (0,1)
	\end{equation}
	and 
	\begin{equation}\label{1137}
		R\in \left(\max \left\{  \pi , L- 4\sqrt{2\varepsilon} \right\}, L \right).
	\end{equation}
    Notice that $R>\pi$.
		
	Let us consider the solution $\phi(x)\in C^2(0,R)$ of the equation
	\begin{equation}\label{1247}
		-\phi'' = \phi (1-a(1-\varepsilon)-\phi) \quad \text{in} \ (0,R), \ \phi(0)=\phi(R)=0. 
	\end{equation}
	Since $R>\pi$ and $1-a(1-\varepsilon)>0$ thanks to the choices in \eqref{1137} and in \eqref{1927}, we get that there exists a solution
	\begin{equation*}
		\phi(x)>0 \qquad \text{for } x\in(0,R).
	\end{equation*}
	Now let us apply Lemma \ref{lemma:psi} in $(0,R)$ with $\varepsilon$ as in \eqref{1927} and 
	\begin{equation}\label{1139}
		C> \frac{4\varepsilon}{L-R},
	\end{equation}
	and let us denote by $\psi$ the solution related to the coefficient $\bar{b}$, so that
	\begin{equation}\label{1248bis}
		-\psi''=\psi (1-\bar{b}\phi(x)-\psi(x)) \quad \text{in} \ (0,R), \ \psi(0)=\psi(R)=1-\varepsilon,
	\end{equation}
	and moreover
	\begin{equation}\label{1301}
		\psi'(0)<-C, \quad \psi'(R)>C, \quad \psi(x)<1-\varepsilon \ \text{in } (0,R).
	\end{equation}
	
	For simplicity, set $M=(L-R)/2$.
	Let us now consider the pair
	\begin{equation*}\label{phi}
	\tilde{\phi}(x)=
	\left\{
	\begin{array}{ll}
	0,  & \text{for } \ -M<x\leq 0, \\
	\phi(x), & \text{for } \ 0<x\leq R, \\
	0,  & \text{for } \ R<x< R+M,
	\end{array}\right.
	\end{equation*}
	and 
	\begin{equation*}\label{psi}
	\tilde{\psi}(x)=
	\left\{
	\begin{array}{ll}
	1-\frac{\varepsilon}{M^2}(x+M)^2,  & \text{for } \ -M<x\leq 0, \\
	\psi(x), & \text{for } \ 0<x\leq R, \\
	1-\frac{\varepsilon}{M^2}(x-R-M)^2,  & \text{for } \ R<x< R+M.
	\end{array}\right.
	\end{equation*}
	Now, we show that the pair $(\tilde{\phi}, \tilde{\psi})$ is a generalised subsolution for system \eqref{sys:barrier}. We start by showing that $(\tilde{\phi}, \tilde{\psi})$ is a subsolution in the three intervals $(-M, 0)$, $(0,R)$, and $(R, R+M)$, by verifying the inequalities as in \eqref{def:subsol}. 
    
    First of all, we have that in $(-M, 0)$, the inequality for $\tilde{\phi}$ is trivial, while for $\tilde{\psi}$ we obtain 	
    \begin{align*}
		 2\frac{\varepsilon}{M^2} &\geq \frac{\varepsilon}{M^2} (x+M)^2 \left(1-\frac{\varepsilon}{M^2}(x+M)^2\right).
	\end{align*}
	This is true since the right-hand side has its maximum equal to $1/4$ and
	\begin{equation*}
		M= \frac{L-R}{2} \leq 2\sqrt{2\varepsilon},
	\end{equation*}
	by the choice in \eqref{1137}.
	A similar argument applies to $(\tilde{\phi}, \tilde{\psi})$  restricted to $(R,R+M)$.
	In $(0,R)$,  $(\tilde{\phi}, \tilde{\psi})$ satisfies \eqref{1247} and \eqref{1248bis}, so in particular, it satisfies the inequalities that prove that the pair constitutes a subsolution in $(0,R)$.

	It is easy to check that  $\tilde{\phi}$ and $\tilde{\psi}$ are continuous in $(-M, R+M)$. 
	Also, we have that 
		\begin{align*}
	\underset{x\to 0^-}{\lim} \tilde{\phi}'(x) &< \underset{x\to 0^+}{\lim} \tilde{\phi}'(x), \\
	\underset{x\to 0^-}{\lim} \tilde{\psi}'(x) &> \underset{x\to 0^+}{\lim} \tilde{\psi}'(x),
	\end{align*}
	thanks to \eqref{1301} and the choice of $C$ in \eqref{1139}. Similar inequalities hold in $x=R$. Thus, $(\tilde{\phi}, \tilde{\psi})$ has the structure demanded by \eqref{def:gensubsol}, therefore it is a generalized subsolution for \eqref{sys:barrier}.

	Moreover, the interval $(-M, R+M)$ has length $L$ and the boundary conditions correspond to the ones given in \eqref{sys:barrier}. Then,  $(\tilde{\phi}, \tilde{\psi})(x-M)$ is a generalised subsolution for \eqref{sys:barrier} in $(0,L)$.
	
	Now, let us take the solution $(y_1,y_2)$ to system \eqref{sys:lv} with initial data $(\tilde{\phi}, \tilde{\psi})(x-M)$ and boundary conditions $(u_1,u_2)\equiv (0,1)$. Then, we have that the solution $(y_1,y_2)$ converges to a steady state solution $(\theta_1, \theta_2)$. Moreover, by the Comparison Principle, we get that 
	\begin{align*}
		0 \not\equiv \tilde{\phi}(x) &\leq \theta_1(x) & \text{in } (0, L),\\
		1 \not\equiv \tilde{\psi}(x) &\geq \theta_2(x) & \text{in } (0, L).
	\end{align*}
	Hence, $(\theta_1, \theta_2)\not \equiv (0,1)$. That is, $(\theta_1, \theta_2)$ is a non trivial solution to \eqref{sys:barrier} for parameters $a$, $L$ and $\bar{b}$. 
	
\end{proof}

\begin{remark}\label{rmk:multid_bbarrier} 
We also notice that the proof of Proposition \ref{thm:bbarrier} can be adapted in higher dimensions when the domain is a ball of radius $l>\ell$ (see Remark \ref{rmk:multid}) by using a paraboloid to extend $\psi$, namely
\begin{equation*}
	\tilde{\psi}(x)=
	\left\{
	\begin{array}{ll}
	\psi(x), & \text{for } \ x\in B_R, \\
	1-\frac{\varepsilon}{M^2}(|x|-R-M)^2,  & \text{for } \ x\in B_{R+M} \setminus B_R,
	\end{array}\right.
	\end{equation*}
where $l=R+M$, $R>\ell$, provided that $M<2\varepsilon$; the computations are very similar to the $1-$dimensional case.

However, when the domain is not a ball, the geometry might make it tricky to construct a subsolution by hand. In this case, extra hypothesis on the domain help in the construction; for example, one can suppose that the domain $\Omega$ contains a ball of radius $l>\ell$, and construct the known subsolution inside the ball, and extend the subsolution to $(0,1)$ outside the ball. However, a characterization of sets admitting or not a barrier (for a suitable $\bar{b}$, that depends on the domain itself) is very hard to achieve.
\end{remark}

We are now ready to present the proof of Theorem \ref{thm:b}.

\begin{proof}[Proof of Theorem \ref{thm:b}]
    Let us define 
    \begin{equation*}
        b^* = \inf \{ b\in (a, \ +\infty) \quad | \quad \text{system \ref{sys:barrier} admits a solution} \}.
    \end{equation*}
    We observe that $b^*<+\infty$ thanks to Proposition \ref{thm:bbarrier}. Then, thanks to the monotonicity property given in Proposition \ref{prop:mono}, a barrier exists for all $b\in (b^*, +\infty)$. 
    When $a>1$,  by Proposition \ref{prop:dependence}, we get that the barrier exists also for $b=b^*$. The second statement of the theorem then follows directly from Theorem~\ref{thm:implies}.

    Let us now prove the first statement of the theorem, which holds for $b\in (1, b^*)$.
    Consider the solution $(y_1,y_2)$ of system \eqref{sys:lv} with boundary conditions $(u_1, u_2)=(0,1)$ for $x=0, \ L$ and $t>0$ and initial datum $(y_1, y_2)\in L^{\infty}((0,L);[0,1])$. Then, thanks to the results in \cite{hirsch1983differential}, $(y_1,y_2)$ converges to a steady state solution $(\theta_1, \theta_2)$ satisfying the first three lines of system \eqref{sys:barrier}.
    By the maximum principle fulfilled by parabolic equations, we have that $0\leq \theta_1, \theta_2 \leq 1$.
    By the definition of $b^*$, there is no barrier for the parameters $a$, $L$, and $b\in (1, b^*)$, that is, system \eqref{sys:barrier} has no solution. Then, we have that $\theta_1\leq 0$ or $\theta_2\geq 1$ in at least one point of $(0,L)$. Again, by the maximum principle, $\theta_1\equiv 0$ or $\theta_2\equiv 1$.  For $(\theta_1, \theta_2)$ to solve the system \eqref{sys:lv}, the only possibility is $(\theta_1, \theta_2)\equiv (0,1)$. So, for $b\in (1, b^*)$, the system is controllable to $(0,1)$ with boundary controls $(u_1, u_2)=(0,1)$ and the proof of the first statement is complete.

\end{proof}

\subsection{Proof of Theorem \ref{thm:a}}

First of all, we prove the existence of barriers for given $b$ and $L$ and for $\bar{a}$ small enough.

\begin{proposition}\label{lemma:a}
	\emph{For any given $b$ and $L>\pi$, and for $\bar{a}\in\left(0,1-\frac{\pi^2}{L^2}\right)$, there exists a barrier, that is, a nontrivial solution to system \eqref{sys:barrier}, for $\bar{a}$, $b$ and $L$.	}
\end{proposition}

\begin{proof}
Let us take $\bar{a}$ such that
\begin{equation}\label{1450}
    \bar{a}< 1-\frac{\pi^2}{L^2}
\end{equation}
and let us define $\ell:=\sqrt{1-\bar{a}}L$. It is easy to check that $\ell>\pi$ thanks to the choice \eqref{1450}. Then, by Lemma \ref{lemma:logistic}, there exists a solution $\Theta_{\ell}>0$ in $(0,\ell)$ to the logistic problem 
\begin{equation*}
    -\Theta_{\ell}''=\Theta_{\ell}(1-\Theta_{\ell}) \quad \text{in} \ (0, \ell), \qquad \Theta_{\ell}(0)=\Theta_{\ell}(\ell)=0.
\end{equation*}

	Let us now consider the function
	\begin{equation*}
	\theta_1 (x) := (1-\bar{a}) \, \Theta_{\ell} \left({x}{\sqrt{1-\bar{a}}}\right),
	\end{equation*}
    defined in $(0,L)$.
	Since
	\begin{equation*}
	-(1-\bar{a})^2\Theta_{\ell}''(x\sqrt{1-\bar{a}})= (1-\bar{a})\Theta_{\ell}(x\sqrt{1-\bar{a}}) \left(1-\bar{a} - (1-\bar{a})\Theta_{\ell}(x\sqrt{1-\bar{a}})  \right) 
	\end{equation*}
    then $\theta_1$ solves the Dirichlet problem
    \begin{equation}\label{eq:true}
      - \theta_1 ''(x)=  \theta_1 (1-\bar{a}-\theta_1) \quad \text{in} \ (0,L), \qquad \theta_1 (0)=\theta_1(L)=0,
    \end{equation}

    Now let us take $\theta_2(x)\equiv 1$ in $[0,L]$. Then, thanks to \eqref{eq:true} the couple $(\theta_1,\theta_2)$ is a subsolution for system \eqref{sys:barrier}. Hence, by the theory of monotone systems (see \cite{sattinger1972monotone}), system \eqref{sys:barrier} admits a solution $(\phi,\psi)$. In fact, by the Comparison Principle of Theorem \ref{thm:cp}, we have that $\phi >\theta_1 >0$ and $\psi <\theta_2=1$ in $(0,L)$, since $(\theta_1,\theta_2)$ is not a solution for the differential equations in system \ref{sys:barrier}.

\end{proof}

Now we are ready to give the proof of Theorem \ref{thm:a}:

\begin{proof}[Proof of Theorem \ref{thm:a}]
Let us take
\begin{equation*}
    a^*=\sup\{ a\in (0,b) \ | \ \text{system \eqref{sys:barrier} admits a solution} \}.
\end{equation*}

	Thanks to Proposition \ref{lemma:a}, $a^*\geq 1-\frac{\pi^2}{L^2}$ because for  $\bar{a}\in\left(0,1-\frac{\pi^2}{L^2}\right)$, system \eqref{sys:barrier} has a non-trivial solution  $(\phi, \psi)$. Then, thanks to Proposition \ref{prop:mono}, a barrier exists for $a\in(0,a^*)$. By Proposition \ref{prop:dependence}  we also have a barrier for $a^*> 1$. Then, the second statement of the theorem is a consequence of Theorem \ref{thm:implies}.

    Let us now prove the first statement of the theorem, which holds for $a\in (a^*,b)$. Let us choose the boundary controls $(u_1, u_2)=(0,1)$ for $x=0, \ L$ and $t>0$ and let us consider the solution $(y_1,y_2)$ of system \eqref{sys:lv} with  initial datum $(y_1, y_2)\in \mathcal{C}^0((0,L);[0,1])$. Then, thanks to the results in \cite{hirsch1983differential}, $(y_1,y_2)$ converges to a steady state solution $(\theta_1, \theta_2)$ fulfilling the first three equations of system \eqref{sys:barrier}.
    By the maximum principle for parabolic equations, we have that $0\leq \theta_1, \theta_2 \leq 1$.
    By the definition of $a^*$, there is no barrier for the parameters $b$, $L$, and $a\in (a^*, b)$, that is, system \eqref{sys:barrier} has no solution. Then, we have that $\theta_1\leq 0$ or $\theta_2\geq 1$ in at least one point of $(0,L)$. Again, by the maximum principle, necessarily, $\theta_1\equiv 0$ or $\theta_2\equiv 1$.  For $(\theta_1, \theta_2)$ to solve the system \eqref{sys:lv}, the only possibility is $(\theta_1, \theta_2)\equiv (0,1)$. So, for $a\in (a^*,b)$, the system is controllable to $(0,1)$ with boundary controls $(u_1, u_2)=(0,1)$ and the proof of the first statement is complete.

\end{proof}

\section{Conclusions and perspectives}

\noindent \textbf{Discussion of the obtained results.}
We have analyzed the controllability of the diffusive Lotka -- Volterra competitive system toward certain constant steady states. Our analysis reveals that, depending on the parameter choices, a variety of distinct dynamical behaviors can emerge:
\begin{itemize}
	\item Controllability to $(0,0)$:
	\begin{itemize}
		\item	Controllability is possible when the spatial domain is sufficiently small, specifically when its length satisfies $L \leq \pi$ (see Theorem~\ref{thm:strong}, part 3), with controls $(u_1,u_2)\equiv (0,0)$. In this case, the null steady state is stable for each equation in the system.
	\item	Controllability fails for larger domains, i.e., when $L > \pi$  (see Theorem~\ref{thm:strong}, part 4). Here, the simultaneous extinction of both populations is obstructed by the instability of the null steady state in each equation, which is caused by the nonlinearity.	
    This result is coherent with the biological interpretation of the problem; in fact, when we have a large domain, eliminating the two species at the boundary is not sufficient to eradicate them inside the domain. 
	\end{itemize}
	\item Controllability to $(1,0)$: the steady state corresponding to the extinction of the least competitive species (in our case $(1,0)$, since we chose the competition coefficient to be $a<b$) can be reached asymptotically from any bounded initial datum (see Theorem \ref{thm:strong}, part 1). 
    This result is also in line with the biological interpretation of the problem, since the stronger species naturally tends to eliminate its competitor.
	\item Controllability to $(0,1)$: regarding the steady state corresponding to the extinction of the most competitive species, we have different behaviours depending on the parameters:
	\begin{itemize}
		\item The steady state $(0,1)$ can be reached from any initial datum when the length of the space-interval $L$ is small (see Theorem \ref{thm:strong}, part 2).
		\item For a general interval of length $L$, it can be reached from some well-chosen initial data (see Theorem \ref{thm:strong}, part 3), which comprises at least constant initial data which are at each point in the basin of attraction (in the sense of a dynamical system) of the equilibrium $(0,1)$.
		\item When the coefficient $b$, which measures the competition of the first population with respect  to the second one, is larger than a threshold (that depends on $a$ and $L$), then the steady state $(0,1)$ cannot be reached from some initial data, see Theorem \ref{thm:b}. This is due to the presence of a steady state for the system, with the same boundary conditions, that attracts the solutions, see also Theorem \ref{thm:implies}. 
	
		\item When the coefficient $a$, which measures the competition of the second population with respect to the first one, is smaller than a threshold (that depends on $b$ and $L$), then the steady state $(0,1)$ cannot be reached from some initial data, see Theorem \ref{thm:a}, againg for the presence of the barrier solution.
        
	\end{itemize}
        What we observe is also coherent with what is observed in nature and it highlights some phenomena in bounded domains (see for example \cite{robertson2017large}). 
        In fact, it is known that, when an alien species reaches an environment occupied by a native species and starts propagating, it is very hard to stop its propagation with solely punctual interventions. 
        Our theorems prove mathematically that eliminating the stronger species in a large domain where it has already started occupying, by only acting on the boundary, is not possible. 
	
	\item Controllability to $(w_1^*, w_2^*)$: 
	as stated in Lemma \ref{lemma:barrier}, a barrier prevents controllability not only to $(0,1)$, but also to $(w_1^*, w_2^*)$. 
	In fact, the barrier is at some point larger than the constant steady state we want to achieve, and no boundary control can make the solution pass below the level of the barrier.
	Hence, all non-controllability results for large $b$ (Theorem \ref{thm:b}) and small $a$ (Theorem \ref{thm:a}) also apply to $(w_1^*, w_2^*)$.  
    Again, this mirrors what we observe in an environment with two strongly competing species where one has an advantage over the other: even with human intervention, it is often not possible to keep the two populations around the coexistence steady state.
\end{itemize}

We also comment on the extension of our results to the case of higher dimensions. 
Notice that the majority of the basic framework that we use, like the Comparison Principle \ref{thm:cp} or the convergence Proposition \ref{prop:dependence}, are valid in any dimension. 
We recall here all the comments on the more technical results:
\begin{itemize}
    \item The proof of Theorem \ref{thm:strong} rely basically on the Comparison Principle \ref{thm:cp} and on Theorem \ref{thm:iida1}, which holds in all dimensions, on the existence of the travelling wave with given sign of the speed, which can be used in the same way, and on Lemma \ref{lemma:logistic}, which can be extended as commented in Remark \ref{rmk:multid}, by adapting the threshold $\pi$ to something else. Therefore, we expect, adjusting the threshold as explained in Remark \ref{rmk:multid}, the same results as the ones for the controllability to $(1,0)$, $(0,0)$, $(0,1)$ and for the non-controllability to $(0,0)$ in Theorem \ref{thm:strong} to hold in dimension $n>1$.
    \item The proof of Theorem \ref{thm:b} relies principally on Proposition \ref{thm:bbarrier}, which can be adapted to higher dimensions for radial domains as explained in Remark \ref{rmk:multid_bbarrier}. 
    In the same remark, we also state that for non-radial domains, extra hypotheses might be required to prove Proposition \ref{thm:bbarrier}. 
    Hence, we expect the result of Theorem \ref{thm:b} -- namely, the controllability or non-controllability to $(0,1)$ depending on a threshold on $b$, due to the existence of a barrier -- to hold in higher dimensions for radial domains under the hypothesis on the radius prescribed by Remark \ref{rmk:multid}, and for star domains under extra hypotheses on the geometry of the domain.  
    \item The proof of Theorem \ref{thm:a} relies principally on Lemma \ref{lemma:logistic}, which can be adapted in the case of higher dimensions as explained in Remark \ref{rmk:multid}. 
    After the proof of Theorem \ref{thm:a}, we comment on the fact that the arguments may be adapted (with heavier computations) to the case of star domains. 
    Therefore, we expect the results of Theorem \ref{thm:a} -- namely, the controllability or non-controllability to $(0,1)$ depending on a threshold on $a$, due to the existence of a barrier -- to hold in higher dimensions for star domains under the hypothesis on the domain prescribed by Remark \ref{rmk:multid}.
\end{itemize}

Moreover, we obtain additional insights into the monotonic dependence of controllability on system parameters (see Proposition~\ref{prop:mono}):
\begin{itemize}
\item If controllability to the steady state $(0,1)$ fails for a given interval of length $L$, then it also fails for all larger intervals.
\item If controllability to the steady state $(0,1)$  fails for a given value of the competition coefficient $b$, then it also fails for all larger values of $b$.
\item If controllability to the steady state $(0,1)$  fails for a given value of the competition coefficient $a$, then it also fails for all smaller values of $a$.
\end{itemize}

This monotonicity property allows us to prove, for any $L> \pi$, the existence of a threshold for both $a$ and $b$. Moreover, it gives us hope to find a threshold for $L$.

\begin{figure}[htbp]
    \centering
    \begin{subfigure}{0.48\textwidth}
        \centering
        \includegraphics[width=\linewidth]{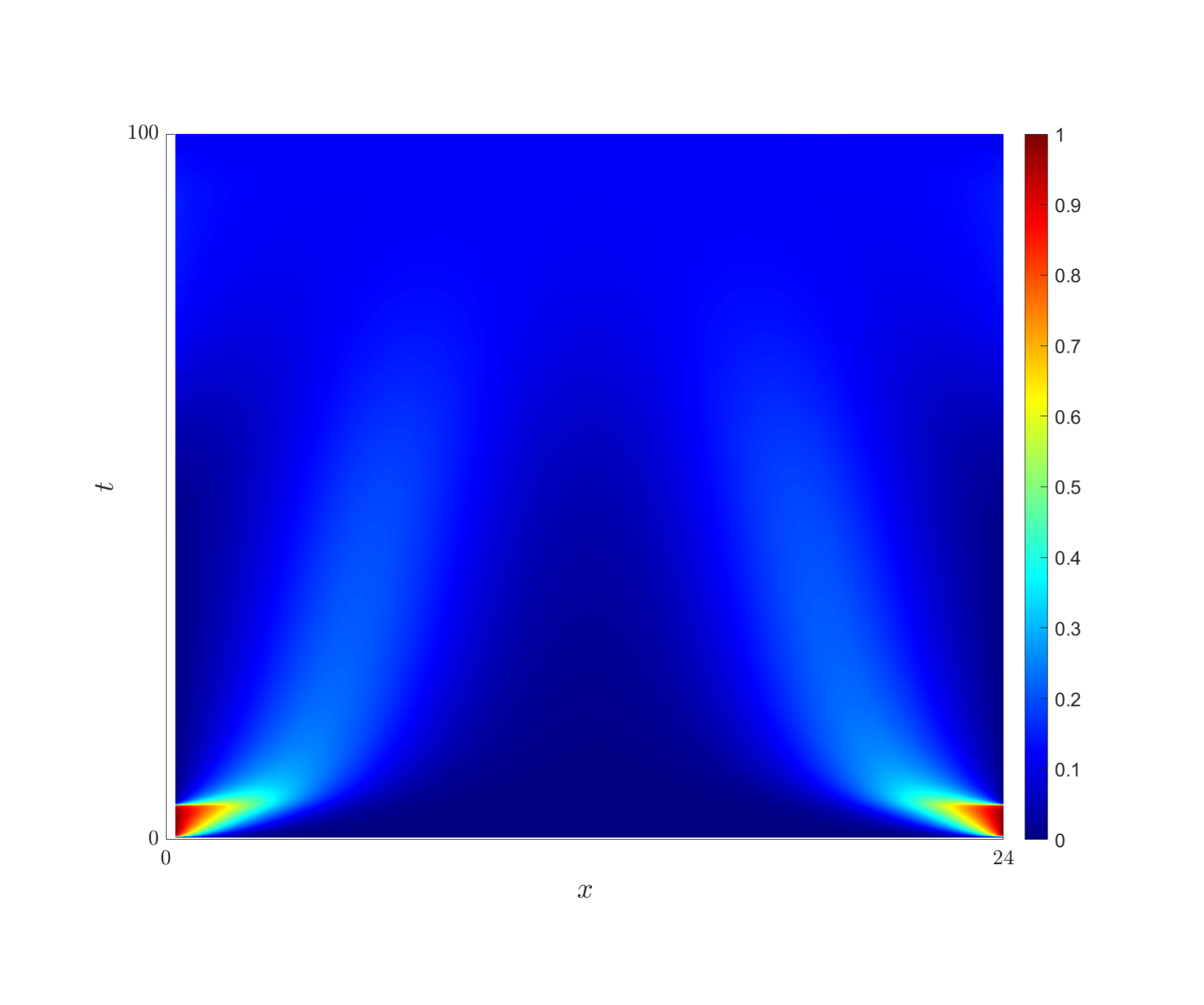}
        \caption{Solution $y_1$}
    \end{subfigure}
    \hfill
    \begin{subfigure}{0.48\textwidth}
        \centering
        \includegraphics[width=\linewidth]{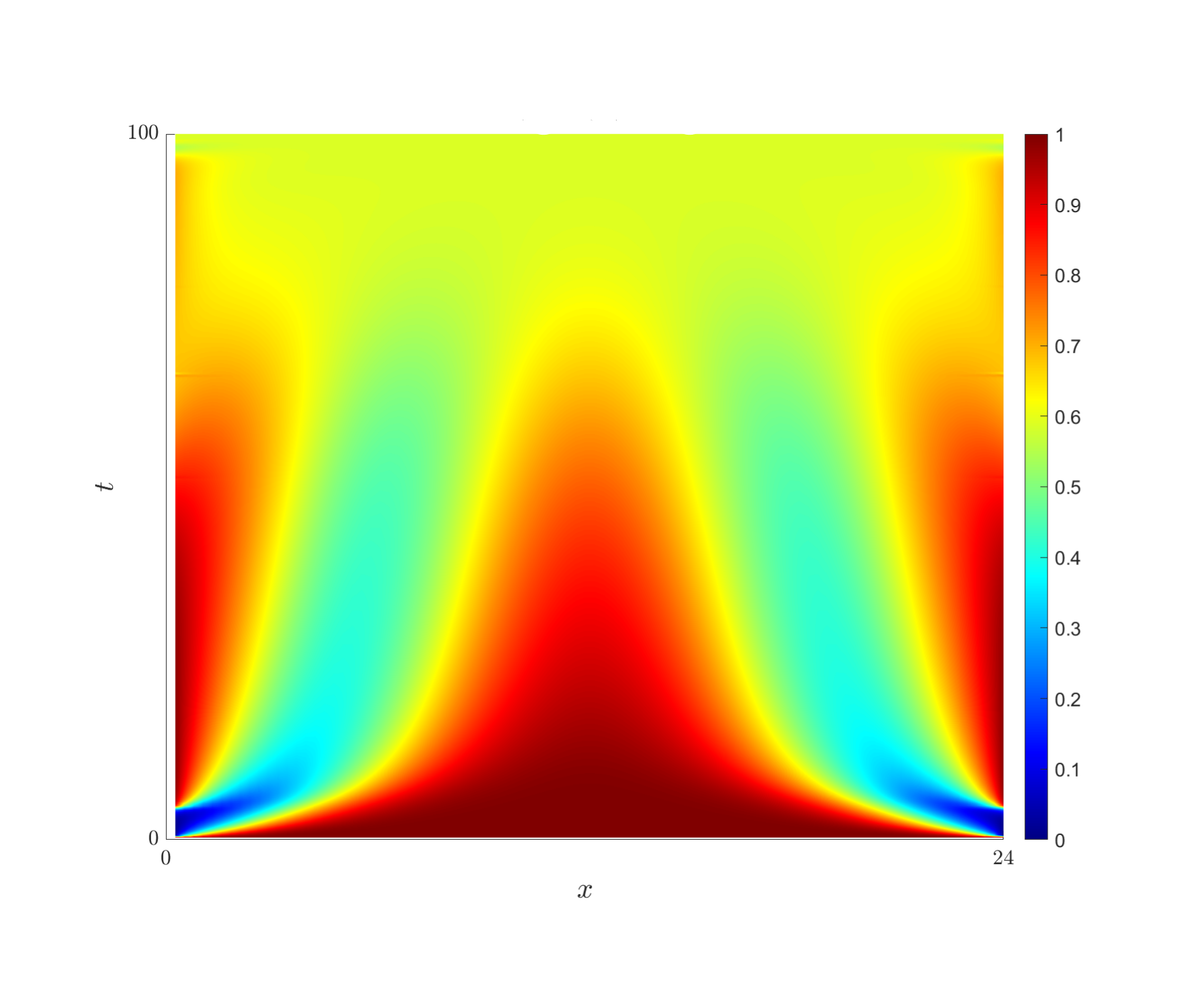}
        \caption{Solution $y_2$}
    \end{subfigure}
    \caption{Controllability to the steady state $(w_1^*, w_2^*) \approx (0.1176, 0.5882)$ from the initial state $(0,1)$ with $L = 24$, $a = 1.5$, $b = 3.5$, and $T = 100$, simulated in \textsc{Matlab} using the \texttt{CasADi} package. The target is well approximated. Similar results have been observed for a wide range of parameters $L$, $a$, and $b$.}
    \label{fig:coex}
\end{figure}

\medskip

\noindent \textbf{Perspectives.}  
We give here final comments that illustrate new research directions opened up by this paper.

\smallskip

As mentioned earlier, we expect the existence of a threshold for the length $L$ of the interval, depending on the parameters $a$ and $b$, such that controllability holds for sufficiently small intervals and fails for larger ones (see Conjecture~\ref{thm:L}). However, proving this rigorously is challenging: the construction of a barrier solution would require detailed knowledge of the shape of solutions, which is currently lacking in the literature. Furthermore, we cannot rule out the possibility that, for certain values of $a$ and $b$, the system may be controllable on any bounded interval (with boundary controls applied at both endpoints), but not on the half-line $\mathbb{R}^+ $ when control is imposed only at $x = 0$. A thorough investigation of this scenario is deferred to future work.

\smallskip

Numerical experiments suggest that it is possible to steer the system from the steady state $(0,1)$ to the coexistence steady state $(w_1^*,w_2^*)$ (see, for instance, Figure~\ref{fig:coex}). In the context of bistable reaction-diffusion equations (see \cite{domenec1}), controllability to a saddle point equilibrium has been achieved by constructing a path of intermediate steady states. A key ingredient in that strategy is the existence of an invariant region in the phase plane.

In contrast, for the Lotka -- Volterra system, the phase plane analysis is significantly more involved due to the higher dimensionality of the system (four variables instead of two) and the absence of a variational structure that could guarantee the existence of suitable path of steady state. Addressing these difficulties remains an open direction for future research.

\smallskip

We did not consider the case $a = b > 1$, where the two species are equally adapted to the environment and compete symmetrically. As shown in Theorem 1.1 of \cite{guo2013sign}, this setting admits a standing wave-- a travelling wave front with zero speed-- which may serve to construct a path of stable steady states by using restrictions of the wave on intervals of length $(0,L)$. This suggests that, using the techniques developed in \cite{dario}, it could be possible to steer the system from any initial data to either of the steady states $(0,1)$ or $(1,0)$, without encountering barrier phenomena.

In contrast, when $a = b = 1$ (the so-called critical case), standing waves do not exist (see \cite{alfaro2023lotka}), and alternative methods would be required.

Regarding the steady state $(0,0)$, when $a=b$, we expect that the conclusions of Theorem \ref{thm:strong} remain valid. The controllability to the coexistence steady state $(w_1^*,w_2^*)$ presents the same difficulties as the case $a\neq b$. 

\smallskip

Finally, we emphasize that the Lotka -- Volterra model serves as a representative example of a  nonlinear reaction-diffusion system where natural constraints on the controls arise. Similar analytical techniques could be applied to other systems with different reaction kinetics or diffusion terms, and warrant further investigation.

\backmatter

\bmhead{Acknowledgements}
The work of the second author was funded by then  ERC Advanced Grant CoDeFeL,  the Grants PID2020-112617GB-C22 KiLearn and TED2021-131390B-I00-DasEl of MINECO and PID2023-146872OB-I00-DyCMaMod of MICIU (Spain),  the Alexander von Humboldt-Professorship program, the European Union's Horizon Europe MSCA project ModConFlex, the Transregio 154 Project ``Mathematical Modelling, Simulation and Optimization Using the Example of Gas Networks'' of the DFG, the AFOSR 24IOE027 project, and the Madrid Government - UAM Agreement for the Excellence of the University Research Staff in the context of the V PRICIT (Regional Programme of Research and Technological Innovation). 



\begin{thebibliography}{34}
\ifx \bisbn   \undefined \def \bisbn  #1{ISBN #1}\fi
\ifx \binits  \undefined \def \binits#1{#1}\fi
\ifx \bauthor  \undefined \def \bauthor#1{#1}\fi
\ifx \batitle  \undefined \def \batitle#1{#1}\fi
\ifx \bjtitle  \undefined \def \bjtitle#1{#1}\fi
\ifx \bvolume  \undefined \def \bvolume#1{\textbf{#1}}\fi
\ifx \byear  \undefined \def \byear#1{#1}\fi
\ifx \bissue  \undefined \def \bissue#1{#1}\fi
\ifx \bfpage  \undefined \def \bfpage#1{#1}\fi
\ifx \blpage  \undefined \def \blpage #1{#1}\fi
\ifx \burl  \undefined \def \burl#1{\textsf{#1}}\fi
\ifx \doiurl  \undefined \def \doiurl#1{\url{https://doi.org/#1}}\fi
\ifx \betal  \undefined \def \betal{\textit{et al.}}\fi
\ifx \binstitute  \undefined \def \binstitute#1{#1}\fi
\ifx \binstitutionaled  \undefined \def \binstitutionaled#1{#1}\fi
\ifx \bctitle  \undefined \def \bctitle#1{#1}\fi
\ifx \beditor  \undefined \def \beditor#1{#1}\fi
\ifx \bpublisher  \undefined \def \bpublisher#1{#1}\fi
\ifx \bbtitle  \undefined \def \bbtitle#1{#1}\fi
\ifx \bedition  \undefined \def \bedition#1{#1}\fi
\ifx \bseriesno  \undefined \def \bseriesno#1{#1}\fi
\ifx \blocation  \undefined \def \blocation#1{#1}\fi
\ifx \bsertitle  \undefined \def \bsertitle#1{#1}\fi
\ifx \bsnm \undefined \def \bsnm#1{#1}\fi
\ifx \bsuffix \undefined \def \bsuffix#1{#1}\fi
\ifx \bparticle \undefined \def \bparticle#1{#1}\fi
\ifx \barticle \undefined \def \barticle#1{#1}\fi
\bibcommenthead
\ifx \bconfdate \undefined \def \bconfdate #1{#1}\fi
\ifx \botherref \undefined \def \botherref #1{#1}\fi
\ifx \url \undefined \def \url#1{\textsf{#1}}\fi
\ifx \bchapter \undefined \def \bchapter#1{#1}\fi
\ifx \bbook \undefined \def \bbook#1{#1}\fi
\ifx \bcomment \undefined \def \bcomment#1{#1}\fi
\ifx \oauthor \undefined \def \oauthor#1{#1}\fi
\ifx \citeauthoryear \undefined \def \citeauthoryear#1{#1}\fi
\ifx \endbibitem  \undefined \def \endbibitem {}\fi
\ifx \bconflocation  \undefined \def \bconflocation#1{#1}\fi
\ifx \arxivurl  \undefined \def \arxivurl#1{\textsf{#1}}\fi
\csname PreBibitemsHook\endcsname

\bibitem[\protect\citeauthoryear{Alfaro and Xiao}{2023}]{alfaro2023lotka}
\begin{barticle}
\bauthor{\bsnm{Alfaro}, \binits{M.}},
\bauthor{\bsnm{Xiao}, \binits{D.}}:
\batitle{Lotka--{V}olterra competition-diffusion system: the critical competition case}.
\bjtitle{Communications in Partial Differential Equations}
\bvolume{48}(\bissue{2}),
\bfpage{182}--\blpage{208}
(\byear{2023})
\end{barticle}
\endbibitem

\bibitem[\protect\citeauthoryear{Blat and Brown}{1984}]{blat1984bifurcation}
\begin{barticle}
\bauthor{\bsnm{Blat}, \binits{J.}},
\bauthor{\bsnm{Brown}, \binits{K.J.}}:
\batitle{Bifurcation of steady-state solutions in predator-prey and competition systems}.
\bjtitle{Proceedings of the Royal Society of Edinburgh Section A: Mathematics}
\bvolume{97},
\bfpage{21}--\blpage{34}
(\byear{1984})
\end{barticle}
\endbibitem

\bibitem[\protect\citeauthoryear{Cantrell and Cosner}{1989}]{cantrell1989diffusive}
\begin{barticle}
\bauthor{\bsnm{Cantrell}, \binits{R.S.}},
\bauthor{\bsnm{Cosner}, \binits{C.}}:
\batitle{Diffusive logistic equations with indefinite weights: population models in disrupted environments}.
\bjtitle{Proceedings of the Royal Society of Edinburgh Section A: Mathematics}
\bvolume{112}(\bissue{3-4}),
\bfpage{293}--\blpage{318}
(\byear{1989})
\end{barticle}
\endbibitem

\bibitem[\protect\citeauthoryear{Cosner and Lazer}{1984}]{cosner1984stable}
\begin{barticle}
\bauthor{\bsnm{Cosner}, \binits{C.}},
\bauthor{\bsnm{Lazer}, \binits{A.C.}}:
\batitle{Stable coexistence states in the {V}olterra--{L}otka competition model with diffusion}.
\bjtitle{SIAM Journal on Applied Mathematics}
\bvolume{44}(\bissue{6}),
\bfpage{1112}--\blpage{1132}
(\byear{1984})
\end{barticle}
\endbibitem

\bibitem[\protect\citeauthoryear{Conti et~al.}{2005}]{conti2005asymptotic}
\begin{barticle}
\bauthor{\bsnm{Conti}, \binits{M.}},
\bauthor{\bsnm{Terracini}, \binits{S.}},
\bauthor{\bsnm{Verzini}, \binits{G.}}:
\batitle{Asymptotic estimates for the spatial segregation of competitive systems}.
\bjtitle{Advances in Mathematics}
\bvolume{195}(\bissue{2}),
\bfpage{524}--\blpage{560}
(\byear{2005})
\end{barticle}
\endbibitem

\bibitem[\protect\citeauthoryear{Gardner}{1982}]{gardner1982existence}
\begin{barticle}
\bauthor{\bsnm{Gardner}, \binits{R.A.}}:
\batitle{Existence and stability of travelling wave solutions of competition models: a degree theoretic approach}.
\bjtitle{Journal of Differential equations}
\bvolume{44}(\bissue{3}),
\bfpage{343}--\blpage{364}
(\byear{1982})
\end{barticle}
\endbibitem

\bibitem[\protect\citeauthoryear{Guo and Lin}{2013}]{guo2013sign}
\begin{barticle}
\bauthor{\bsnm{Guo}, \binits{J.-S.}},
\bauthor{\bsnm{Lin}, \binits{Y.-C.}}:
\batitle{The sign of the wave speed for the {L}otka-{V}olterra competition-diffusion system}.
\bjtitle{Communications on Pure \& Applied Analysis}
\bvolume{12}(\bissue{5}),
\bfpage{2083}
(\byear{2013})
\end{barticle}
\endbibitem

\bibitem[\protect\citeauthoryear{Girardin and Lam}{2019}]{girardin2019invasion}
\begin{barticle}
\bauthor{\bsnm{Girardin}, \binits{L.}},
\bauthor{\bsnm{Lam}, \binits{K.-Y.}}:
\batitle{Invasion of open space by two competitors: spreading properties of monostable two-species competition-diffusion systems}.
\bjtitle{Proceedings of the London Mathematical Society}
\bvolume{119}(\bissue{5}),
\bfpage{1279}--\blpage{1335}
(\byear{2019})
\end{barticle}
\endbibitem

\bibitem[\protect\citeauthoryear{Guo et~al.}{2019}]{guo2019sign}
\begin{barticle}
\bauthor{\bsnm{Guo}, \binits{J.-S.}},
\bauthor{\bsnm{Nakamura}, \binits{K.-I.}},
\bauthor{\bsnm{Ogiwara}, \binits{T.}},
\bauthor{\bsnm{Wu}, \binits{C.-H.}}:
\batitle{The sign of traveling wave speed in bistable dynamics}.
\bjtitle{Discrete and Continuous Dynamical Systems}
\bvolume{40}(\bissue{6}),
\bfpage{3451}--\blpage{3466}
(\byear{2019})
\end{barticle}
\endbibitem

\bibitem[\protect\citeauthoryear{Hirsch}{1983}]{hirsch1983differential}
\begin{barticle}
\bauthor{\bsnm{Hirsch}, \binits{M.W.}}:
\batitle{Differential equations and convergence almost everywhere in strongly monotone semiflows}.
\bjtitle{Contemp. Math}
\bvolume{17},
\bfpage{267}--\blpage{285}
(\byear{1983})
\end{barticle}
\endbibitem

\bibitem[\protect\citeauthoryear{Hirsch}{1988}]{hirsch1988systems}
\begin{barticle}
\bauthor{\bsnm{Hirsch}, \binits{M.W.}}:
\batitle{Systems of differential equations which are competitive or cooperative: Iii. competing species}.
\bjtitle{Nonlinearity}
\bvolume{1}(\bissue{1}),
\bfpage{51}
(\byear{1988})
\end{barticle}
\endbibitem

\bibitem[\protect\citeauthoryear{Iida et~al.}{1998}]{iida1998diffusion}
\begin{barticle}
\bauthor{\bsnm{Iida}, \binits{M.}},
\bauthor{\bsnm{Muramatsu}, \binits{T.}},
\bauthor{\bsnm{Ninomiya}, \binits{H.}},
\bauthor{\bsnm{Yanagida}, \binits{E.}}:
\batitle{Diffusion-induced extinction of a superior species in a competition system}.
\bjtitle{Japan journal of industrial and applied mathematics}
\bvolume{15}(\bissue{2}),
\bfpage{233}--\blpage{252}
(\byear{1998})
\end{barticle}
\endbibitem

\bibitem[\protect\citeauthoryear{Kan-On}{1996}]{kan1996existence}
\begin{barticle}
\bauthor{\bsnm{Kan-On}, \binits{Y.}}:
\batitle{Existence of standing waves for competition-diffusion equations}.
\bjtitle{Japan Journal of Industrial and Applied Mathematics}
\bvolume{13}(\bissue{1}),
\bfpage{117}--\blpage{133}
(\byear{1996})
\end{barticle}
\endbibitem

\bibitem[\protect\citeauthoryear{Kan-On}{1997}]{kan1997fisher}
\begin{barticle}
\bauthor{\bsnm{Kan-On}, \binits{Y.}}:
\batitle{Fisher wave fronts for the {L}otka-{V}olterra competition model with diffusion}.
\bjtitle{Nonlinear Analysis: Theory, methods \& Applications}
\bvolume{28}(\bissue{1}),
\bfpage{145}--\blpage{164}
(\byear{1997})
\end{barticle}
\endbibitem

\bibitem[\protect\citeauthoryear{Le~Balc'H}{2019}]{balch}
\begin{barticle}
\bauthor{\bsnm{Le~Balc'H}, \binits{K.}}:
\batitle{Null-controllability of two species reaction-diffusion system with nonlinear coupling: a new duality method}.
\bjtitle{SIAM Journal on Control and Optimization}
\bvolume{57}(\bissue{4}),
\bfpage{2541}--\blpage{2573}
(\byear{2019})
\end{barticle}
\endbibitem

\bibitem[\protect\citeauthoryear{Ma et~al.}{2019}]{ma2019speed}
\begin{barticle}
\bauthor{\bsnm{Ma}, \binits{M.}},
\bauthor{\bsnm{Huang}, \binits{Z.}},
\bauthor{\bsnm{Ou}, \binits{C.}}:
\batitle{Speed of the travelling wave for the bistable {L}otka--{V}olterra competition model}.
\bjtitle{Nonlinearity}
\bvolume{32}(\bissue{9}),
\bfpage{3143}
(\byear{2019})
\end{barticle}
\endbibitem

\bibitem[\protect\citeauthoryear{Matano and Mimura}{1983}]{matano1983pattern}
\begin{barticle}
\bauthor{\bsnm{Matano}, \binits{H.}},
\bauthor{\bsnm{Mimura}, \binits{M.}}:
\batitle{Pattern formation in competition-diffusion systems in nonconvex domains}.
\bjtitle{Publications of the Research Institute for Mathematical Sciences}
\bvolume{19}(\bissue{3}),
\bfpage{1049}--\blpage{1079}
(\byear{1983})
\end{barticle}
\endbibitem

\bibitem[\protect\citeauthoryear{Ma et~al.}{2020}]{ma2020bistable}
\begin{barticle}
\bauthor{\bsnm{Ma}, \binits{M.}},
\bauthor{\bsnm{Zhang}, \binits{Q.}},
\bauthor{\bsnm{Yue}, \binits{J.}},
\bauthor{\bsnm{Ou}, \binits{C.}}:
\batitle{Bistable wave speed of the {L}otka-{V}olterra competition model}.
\bjtitle{Journal of Biological Dynamics}
\bvolume{14}(\bissue{1}),
\bfpage{608}--\blpage{620}
(\byear{2020})
\end{barticle}
\endbibitem

\bibitem[\protect\citeauthoryear{Ninomiya}{1997}]{ninomiya1997separatrices}
\begin{bchapter}
\bauthor{\bsnm{Ninomiya}, \binits{H.}}:
\bctitle{Separatrices of competition-diffusion equations}.
In: \bbtitle{Reaction-diffusion Equations And Their Applications And Computational Aspects-Proceedings Of The China-japan Symposium},
p. \bfpage{118}
(\byear{1997}).
\bcomment{World Scientific}
\end{bchapter}
\endbibitem

\bibitem[\protect\citeauthoryear{Py{\v{s}}ek et~al.}{2020}]{pyvsek2020scientists}
\begin{barticle}
\bauthor{\bsnm{Py{\v{s}}ek}, \binits{P.}},
\bauthor{\bsnm{Hulme}, \binits{P.E.}},
\bauthor{\bsnm{Simberloff}, \binits{D.}},
\bauthor{\bsnm{Bacher}, \binits{S.}},
\bauthor{\bsnm{Blackburn}, \binits{T.M.}},
\bauthor{\bsnm{Carlton}, \binits{J.T.}},
\bauthor{\bsnm{Dawson}, \binits{W.}},
\bauthor{\bsnm{Essl}, \binits{F.}},
\bauthor{\bsnm{Foxcroft}, \binits{L.C.}},
\bauthor{\bsnm{Genovesi}, \binits{P.}}, \betal:
\batitle{Scientists' warning on invasive alien species}.
\bjtitle{Biological Reviews}
\bvolume{95}(\bissue{6}),
\bfpage{1511}--\blpage{1534}
(\byear{2020})
\end{barticle}
\endbibitem

\bibitem[\protect\citeauthoryear{Pouchol et~al.}{2019}]{pouchol2019phase}
\begin{barticle}
\bauthor{\bsnm{Pouchol}, \binits{C.}},
\bauthor{\bsnm{Tr{\'e}lat}, \binits{E.}},
\bauthor{\bsnm{Zuazua}, \binits{E.}}:
\batitle{Phase portrait control for 1d monostable and bistable reaction--diffusion equations}.
\bjtitle{Nonlinearity}
\bvolume{32}(\bissue{3}),
\bfpage{884}
(\byear{2019})
\end{barticle}
\endbibitem

\bibitem[\protect\citeauthoryear{Protter and Weinberger}{1984}]{protter2012maximum}
\begin{botherref}
\oauthor{\bsnm{Protter}, \binits{M.H.}},
\oauthor{\bsnm{Weinberger}, \binits{H.F.}}:
Maximum principles in differential equations.
Springer-Verlag
(1984)
\end{botherref}
\endbibitem

\bibitem[\protect\citeauthoryear{Pighin and Zuazua}{2018}]{dario}
\begin{botherref}
\oauthor{\bsnm{Pighin}, \binits{D.}},
\oauthor{\bsnm{Zuazua}, \binits{E.}}:
Controllability under positivity constraints of semilinear heat equations.
Mathematical Control and Related Fields
(2018)
\end{botherref}
\endbibitem

\bibitem[\protect\citeauthoryear{Robertson et~al.}{2017}]{robertson2017large}
\begin{barticle}
\bauthor{\bsnm{Robertson}, \binits{P.A.}},
\bauthor{\bsnm{Adriaens}, \binits{T.}},
\bauthor{\bsnm{Lambin}, \binits{X.}},
\bauthor{\bsnm{Mill}, \binits{A.}},
\bauthor{\bsnm{Roy}, \binits{S.}},
\bauthor{\bsnm{Shuttleworth}, \binits{C.M.}},
\bauthor{\bsnm{Sutton-Croft}, \binits{M.}}:
\batitle{The large-scale removal of mammalian invasive alien species in {N}orthern {E}urope}.
\bjtitle{Pest management science}
\bvolume{73}(\bissue{2}),
\bfpage{273}--\blpage{279}
(\byear{2017})
\end{barticle}
\endbibitem

\bibitem[\protect\citeauthoryear{Ruiz-Balet and Zuazua}{2020}]{domenec1}
\begin{barticle}
\bauthor{\bsnm{Ruiz-Balet}, \binits{D.}},
\bauthor{\bsnm{Zuazua}, \binits{E.}}:
\batitle{Control under constraints for multi-dimensional reaction-diffusion monostable and bistable equations}.
\bjtitle{Journal de Math{\'e}matiques Pures et Appliqu{\'e}es}
\bvolume{143},
\bfpage{345}--\blpage{375}
(\byear{2020})
\end{barticle}
\endbibitem

\bibitem[\protect\citeauthoryear{Sattinger}{1972}]{sattinger1972monotone}
\begin{barticle}
\bauthor{\bsnm{Sattinger}, \binits{D.H.}}:
\batitle{Monotone methods in nonlinear elliptic and parabolic boundary value problems}.
\bjtitle{Indiana University Mathematics Journal}
\bvolume{21}(\bissue{11}),
\bfpage{979}--\blpage{1000}
(\byear{1972})
\end{barticle}
\endbibitem

\bibitem[\protect\citeauthoryear{Shim}{2002}]{shim2002domains}
\begin{barticle}
\bauthor{\bsnm{Shim}, \binits{S.-A.}}:
\batitle{Domains of attraction of competition-diffusion systems}.
\bjtitle{Funkcialaj EKvacioj Serio Internacia}
\bvolume{45}(\bissue{1}),
\bfpage{103}--\blpage{122}
(\byear{2002})
\end{barticle}
\endbibitem

\bibitem[\protect\citeauthoryear{Shigesada et~al.}{1984}]{shigesada1984effects}
\begin{barticle}
\bauthor{\bsnm{Shigesada}, \binits{N.}},
\bauthor{\bsnm{Kawasaki}, \binits{K.}},
\bauthor{\bsnm{Teramoto}, \binits{E.}}:
\batitle{The effects of interference competition on stability, structure and invasion of a multi-species system}.
\bjtitle{Journal of Mathematical Biology}
\bvolume{21},
\bfpage{97}--\blpage{113}
(\byear{1984})
\end{barticle}
\endbibitem

\bibitem[\protect\citeauthoryear{Smillie}{1984}]{smillie1984competitive}
\begin{barticle}
\bauthor{\bsnm{Smillie}, \binits{J.}}:
\batitle{Competitive and cooperative tridiagonal systems of differential equations}.
\bjtitle{SIAM journal on mathematical analysis}
\bvolume{15}(\bissue{3}),
\bfpage{530}--\blpage{534}
(\byear{1984})
\end{barticle}
\endbibitem

\bibitem[\protect\citeauthoryear{Smith}{1988}]{smith1988systems}
\begin{barticle}
\bauthor{\bsnm{Smith}, \binits{H.L.}}:
\batitle{Systems of ordinary differential equations which generate an order preserving flow. {A} survey of results}.
\bjtitle{SIAM review}
\bvolume{30}(\bissue{1}),
\bfpage{87}--\blpage{113}
(\byear{1988})
\end{barticle}
\endbibitem

\bibitem[\protect\citeauthoryear{Smith}{2008}]{Smith2008MonotoneDS}
\begin{bbook}
\bauthor{\bsnm{Smith}, \binits{H.L.}}:
\bbtitle{Monotone Dynamical Systems: an Introduction to the Theory of Competitive and Cooperative Systems}
vol. \bseriesno{41}.
\bpublisher{American Mathematical Soc.}, \blocation{Providence}
(\byear{2008})
\end{bbook}
\endbibitem

\bibitem[\protect\citeauthoryear{Sonego and Zuazua}{2024}]{sonego2024control}
\begin{botherref}
\oauthor{\bsnm{Sonego}, \binits{M.}},
\oauthor{\bsnm{Zuazua}, \binits{E.}}:
Control of a {L}otka-{V}olterra system with weak competition.
arXiv preprint arXiv:2409.20279
(2024)
\end{botherref}
\endbibitem

\bibitem[\protect\citeauthoryear{Tang and Fife}{1980}]{tang1980propagating}
\begin{barticle}
\bauthor{\bsnm{Tang}, \binits{M.M.}},
\bauthor{\bsnm{Fife}, \binits{P.C.}}:
\batitle{Propagating fronts for competing species equations with diffusion}.
\bjtitle{Archive for Rational Mechanics and Analysis}
\bvolume{73}(\bissue{1}),
\bfpage{69}--\blpage{77}
(\byear{1980})
\end{barticle}
\endbibitem

\bibitem[\protect\citeauthoryear{Terracini et~al.}{2019}]{terracini2019spiraling}
\begin{barticle}
\bauthor{\bsnm{Terracini}, \binits{S.}},
\bauthor{\bsnm{Verzini}, \binits{G.}},
\bauthor{\bsnm{Zilio}, \binits{A.}}:
\batitle{Spiraling asymptotic profiles of competition-diffusion systems}.
\bjtitle{Communications on Pure and Applied Mathematics}
\bvolume{72}(\bissue{12}),
\bfpage{2578}--\blpage{2620}
(\byear{2019})
\end{barticle}
\endbibitem

\end{thebibliography}


\end{document}